\documentclass[onefignum,onetabnum]{siamart171218}

\usepackage{amsfonts}
\usepackage[margin=1in]{geometry}
\usepackage{subcaption}

\newsiamremark{remark}{Remark}
\newcommand{\mean}{m}
\newcommand{\pmean}{\widehat{m}}
\newcommand{\pp}{\widehat{\theta}}
\newcommand{\py}{\widehat{y}}
\newcommand{\ppy}{\widehat{y}}
\newcommand{\Cov}{C}
\newcommand{\pCov}{\widehat{C}}
\newcommand{\pf}{p}

\newcommand{\N}{\mathcal{N}}
\newcommand{\G}{\mathcal{G}}
\newcommand{\F}{\mathcal{F}}
\newcommand{\E}{\mathbb{E}}
\newcommand{\I}{\mathbb{I}}
\newcommand{\R}{\mathbb{R}}
\newcommand{\Z}{\mathbb{Z}}

\newcommand{\bigO}{\mathcal{O}}

\title{Unscented Kalman Inversion: Efficient Gaussian Approximation to the Posterior Distribution}

\author{Daniel~Z.~Huang\thanks{California Institute of Technology, Pasadena, CA 
  (\email{dzhuang@caltech.edu}).}
\and Jiaoyang~Huang\thanks{New York University, New York, NY
  (\email{jh4427@nyu.edu}).}
  }

\begin{document}

\maketitle

% REQUIRED
\begin{abstract}
The unscented Kalman inversion~(UKI) method presented in~\cite{UKI} is a general derivative-free approach for the inverse problem. 
UKI is particularly suitable for inverse problems where the forward model is given as a black box and may not be differentiable.
The regularization strategies, convergence property, and speed-up strategies~\cite{UKI,UKI2} of the UKI are thoroughly studied, and the method is capable of handling noisy observation data and solving chaotic inverse problems. 
In this paper, we study the uncertainty quantification capability of the UKI.  
We propose a modified UKI, which allows to well approximate the mean and covariance of the posterior distribution for well-posed inverse problems with large observation data.
Theoretical guarantees for both linear and nonlinear inverse problems are presented.
Numerical results, including learning  of  permeability  parameters  in  subsurface flow and of the Navier-Stokes initial condition from solution data  at  positive times are presented. 
The results obtained by the UKI require only $\bigO(10)$ iterations, and 
match well with the expected results obtained by the Markov Chain Monte Carlo method.

\end{abstract}

% REQUIRED

% REQUIRED

\section{Introduction}
Inverse problems are ubiquitous in engineering and science applications. These include, to name only a few, global climate model calibration~\cite{sen2013global,schneider2017earth,dunbar2020calibration}, material constitutive relation calibration~\cite{huang2020learning,xu2021learning,avery2020computationally}, seismic inversion in geophysics~\cite{russell1988introduction,bunks1995multiscale}, and medical tomography~\cite{mukherjee2020learned,trigo2004electrical}.
These problems may feature multiple scales and may include chaotic and turbulent phenomena, and hence the forward models are very expensive to evaluate. Moreover, the observation data are noisy and uncertainty quantification is important.

Inverse problems can be formulated as recovering unknown parameters $\theta \in \R^{N_{\theta}}$ from the noisy observation $y \in \R^{N_y}$, as following
\begin{equation}
\label{eq:KI}
    y = \G(\theta) + \eta,
\end{equation}
where $\G$ denotes a forward operator mapping parameters to observations, and $\eta$ denotes the observational error, which might depend on $\theta$ or $y$.
An estimated Gaussian distribution of $\eta \sim \N(0,\Sigma_{\eta})$ is given.

From optimization viewpoint, the inverse problem can be formulated to the following nonlinear least-square optimization problem~\cite{engl1996regularization}:
% \begin{equation}
% \label{eq:KI2}
%     \min_{\theta} \Phi(\theta, y) = \frac{1}{2}\lVert\Sigma_{\eta}^{-\frac{1}{2}}(y - \G(\theta)) \rVert^2.
% \end{equation}
\begin{subequations}
\label{eq:KI2}
\begin{align}
\Phi_R(\theta,y) &:= \Phi(\theta,y)+\frac{1}{2}\lVert\Sigma_{0}^{-\frac{1}{2}}(\theta - r_0) \rVert^2,\\
    \Phi(\theta,y) &:= \frac{1}{2}\lVert\Sigma_{\eta}^{-\frac{1}{2}}(y - \G(\theta)) \rVert^2,\label{eq:KI2-b}
\end{align}
\end{subequations}
where $\Phi_R$ and $\Phi$ are regularized and non-regularized objective functions, respectively. $\Sigma_{\eta} \succ 0$ is strictly positive-definite and normalizes the model-data misfit.
$r_0$ and $\Sigma_{0}\succ 0$ encode
prior mean and covariance information about $\theta$.
From Bayesian viewpoint, $\theta$ and $y$ are treated as random variables, and the inverse problem can be formulated as posterior distribution approximation problem~\cite{kaipio2006statistical, dashti2013bayesian}:
\begin{equation}
\label{eq:KI3}
    \mu(d\theta) = \frac{1}{Z(y)} e^{-\Phi(\theta , y) } \mu_0(d\theta),
\end{equation}
where $\mu_0(d\theta)$ is the prior and $Z(y)$ is the normalization constant:
\begin{equation}
    Z(y) =  \int e^{-\Phi(\theta , y) } \mu_0(d\theta).
\end{equation}
The optimization viewpoint and the Bayesian viewpoint are linked via the fact that 
the minimizer of the regularized nonlinear least-square optimization problem $\Phi_R(\theta,y)$ coincides with the maximum a posterior~(MAP) estimator of \eqref{eq:KI3} 
and
the minimizer of the nonlinear least-square optimization problem $\Phi(\theta,y)$ coincides with the MAP estimator of \eqref{eq:KI3} with an uninformative prior, specifically the improper uniform prior. 
In the present study, we focus on inverse problems with large observation data but possibly without strong beliefs about $\theta$, and therefore, it is common to use uninformative priors to let the data speak for itself.

A useful approach to solve inverse problems, at the intersection of optimization and Bayesian viewpoints,  is to pair the parameter-to-data map~\eqref{eq:KI}  with a stochastic
dynamical system for the parameter, and then employ techniques from filtering to estimate the
parameter given the data. 
Consider the following stochastic dynamical system,\footnote{In the present study, we focus on the  dynamics without regularization, which is equivalent to setting $\alpha=1$ in~\cite{UKI, UKI2}.}
\begin{subequations}
\label{eq:dynamics}
  \begin{align}
  &\textrm{evolution:}    &&\theta_{n+1} = \theta_{n} +  \omega_{n+1}, &&\omega_{n+1} \sim \N(0,\Sigma_{\omega}) \label{eq:evolve}\\
  &\textrm{observation:}  &&y_{n+1} = \G(\theta_{n+1}) + \nu_{n+1}, &&\nu_{n+1} \sim \N(0,\Sigma_{\nu})
\end{align}
\end{subequations}
where $\theta_{n+1}$ is the unknown state vector, and $y_{n+1} =y$ is the observation, the artificial evolution error $\omega_{n+1}$ and artificial observation error $\nu_{n+1}$ are mutually independent, zero-mean Gaussian sequences with covariances $\Sigma_{\omega}$ and $\Sigma_{\nu}$, respectively. 
The Kalman inversion methodology, which takes advantage of Kalman filtering~\cite{kalman1960new,sorenson1985kalman,evensen1994sequential,julier1995new,arasaratnam2009cubature}, originates in the papers~\cite{singhal1989training, puskorius1991decoupled, wan2000unscented, chen2012ensemble,iglesias2013ensemble,emerick2013investigation}.
This methodology is widely used as a non-intrusive optimization method for parameter estimation.
However, as for Bayesian inference, due to the Gaussian ansatz in Kalman filtering, the capability of Kalman inversions to quantify uncertainties or approximate posterior distribution remains unclarified.  
We focus on understanding and improving the capability of Kalman inversions, specifically the UKI, for Bayesian inference.
We find \textbf{adaptively} updating $\Sigma_{\omega}$ in the evolution dynamics~\cref{eq:dynamics}, which leads to an evolving stochastic dynamical system, enables UKI to provide accurate Gaussian approximations to the posterior distribution.

\subsection{Literature Review}

The Kalman filter~\cite{kalman1960new} and its variants, including but not limit to extended Kalman filter~\cite{sorenson1985kalman}, ensemble Kalman filter~\cite{evensen1994sequential}, unscented Kalman filter~\cite{julier1995new,wan2000unscented}, and cubature Kalman filter~\cite{arasaratnam2009cubature} are developed to sequentially update the probability distribution of states in  partially  observed  dynamics.
Kalman filtering is a two-step procedure:  the prediction step, where the state is computed forward in time;  
the analysis step, where the state and its uncertainty are corrected to take into account the observation.
In the analysis step, Kalman filters use Gaussian ansatz to formulate Kalman gain to assimilate the observation and update the distribution, which is valid only for linear state estimation problems, where the probability distribution of states remains Gaussian given Gaussian priors.
Therefore, Kalman filters do not  produce correct estimate of posterior distribution for non-Gaussian problems~\cite{zafari2005assessing}.
However, numerous applications of Kalman filters, including weather forecasts 
~\cite{evensen1994sequential,anderson2001ensemble,bishop2001adaptive} and guidance, navigation, and control of vehicles~\cite{sorenson1985kalman,julier1995new, arasaratnam2009cubature} demonstrate empirically that Kalman filters can not only calibrate the model predication, but also provides uncertainty information for nonlinear state estimation problems, specifically the mean and covariance estimation.
To get better posterior estimation, particle filters~\cite{smith2013sequential} are needed, where the distribution is represented by a large number of random samples.
Most of particle filters consist of prediction step and analysis step.  In the analysis step, these samples and their associated weights are corrected with the information provided by the observation. 
The standard particle filter corrects the weights of each sample following Bayes' theorem.  Different correction methods are proposed, which lead to the random maximum  likelihood method~\cite{kitanidis1995quasi,oliver1996conditioning}, the nonparametric ensemble transform method~\cite{reich2013nonparametric},  and etc.

Filter methods can also be used to estimate the parameters, 
where the  filters are recursively applied to an artificial stochastic dynamical system generally with identity state transition matrix.
This leads to different inversion methods, to name only a few, sequential Monte Carlo sampler~(SMC)~\cite{del2006sequential, kantas2014sequential, beskos2015sequential},  
random maximum  likelihood method~\cite{oliver1996conditioning,chen2012ensemble}, ensemble Kalman inversion~\cite{evensen2009ensemble, gu2006ensemble, iglesias2013ensemble},  unscented Kalman inversion~\cite{wan2000unscented,UKI, UKI2}.
Besides estimating the parameters by minimizing the nonlinear least square problems~\eqref{eq:KI2},  these inversion methods are derived in the Bayesian framework, and therefore,  have the potential to deliver sensitivity and uncertainty information.

This paper mainly focuses on Kalman inversion methodology, which is a general derivative-free approach to solving the inverse problem. 
For inverse problems in general, they are attractive because they are derivative-free, and hence introduce a significant flexibility in the forward solver design.
Therefore, Kalman inversions are suitable for complex multiphysics problems requiring coupling of different solvers~\cite{huang2018simulation,huang2019high,huang2020high,huang2020modeling,adcroft2019gfdl} and methods containing discontinuities
~(i.e., immersed/embedded boundary method~\cite{peskin1977numerical,berger2012progress,huang2018family,huang2020embedded} and adaptive mesh refinement~\cite{berger1989local,borker2019mesh}).
So far, the Kalman inversions have mainly been used for parameter estimation~\cite{iglesias2013ensemble,iglesias2016regularizing, schillings2017analysis,schillings2018convergence, iglesias2020adaptive,chada2020iterative} rather than for Bayesian inference for quantifying uncertainty. 
Due to the Gaussian ansatz and the iterative nature, 
Kalman inversions generally do not converge to the posterior distribution. Specifically, the mean converges but not the covariance in the nonlinear case. Negative numerical evidences are reported in~\cite{ernst2015analysis,garbuno2020interacting}. 
In the present work, we improve the UKI by modifying the stochastic dynamical system through iteratively and adaptively updating the artificial evolution error covariance  matrix. 
And we demonstrate both theoretically and numerically that equipping with the modified stochastic dynamical system,
the Kalman inversion, specifically the UKI, is able to provide good Gaussian approximations to the posterior distribution with an uninformative prior for both linear and nonlinear inverse problems under certain regularity assumptions.

\subsection{Our Contributions}

Our main contribution is the development of theoretical and numerical understanding about UKI for Bayesian inference for quantifying uncertainty.

\begin{itemize}
    \item We modify the underlying stochastic dynamical system of Kalman inversions proposed in~\cite{UKI, UKI2} by iteratively and adaptively updating the artificial evolution covariance matrix, which enables better approximating the posterior covariance matrix.
    \item With this modification, for linear inverse problems, we prove the exponential convergence of the UKI. The mean converges to a minimizer of $\Phi$ and the precision matrix converges to the posterior precision matrix with an uninformative prior.
    \item With this modification, for nonlinear inverse problems, we prove the UKI approximates well with the mean and covariance of the posterior distribution with an uninformative prior, when the forward operator $\G$ is bijective and satisfies certain regularity conditions.  
    \item We demonstrate on the inverse problems studied in the present work
    that the UKI delivers similar posterior mean and covariance comparing to Markov Chain Monte Carlo method and 
    is more efficient than other derivative-free ensemble methods.
    \item For  the  UKI,  all  tests  converge  within $\bigO(10)$  iterations  with  no  empirical variance inflation or early stopping needed.
\end{itemize}

The remainder of the paper is organized as follows. 
In \cref{sec:Alg}, an overview of Kalman inversion algorithm and the modified UKI are presented.
In \cref{sec:Theorem}, theoretical results for the modified UKI are presented.
Numerical applications are provided in \cref{sec:app}, that empirically confirm the theories and demonstrate the effectiveness of the UKI for Bayesian inference for quantifying uncertainty.

\section{The Algorithm}
\label{sec:Alg}
Kalman inversion methods pair the parameter-to-data map~\eqref{eq:KI} with the stochastic
dynamical system~\eqref{eq:dynamics}.
Let denote  $Y_{n} := \{y_1, y_2,\cdots , y_{n}\}$, the observation set at time $n$~(Although $Y_{n}  = y$).
The techniques from Kalman filtering are employed to estimate the distribution  $\mu_{n}$, the conditional distribution of $\theta_{n}|Y_{n}$.
And the updating of $\{\mu_n\}$ is through the prediction and analysis steps~\cite{reich2015probabilistic,law2015data}:
$\mu_n \mapsto \hat{\mu}_{n+1}$, and then $\hat{\mu}_{n+1} \mapsto \mu_{n+1}$, where $\hat{\mu}_{n+1}$ is the distribution of $\theta_{n+1}|Y_n$.
The approximated distribution of $\mu_{n}$ is assumed to be Gaussian, and hence, Kalman inversion methods are Gaussian approximation algorithms.
The  Gaussian  approximation  is  good  when  there  are  enough  observation  data.  Since Bernstein-von Mises theorem~\cite{le2012asymptotics,van2000asymptotic,freedman1999wald,lu2017gaussian} states the posterior distribution becomes asymptotically a multivariate normal distribution under certain regularity conditions.

\subsection{Gaussian Approximation}
\label{sec:KF_Intro}
This conceptual algorithm maps Gaussians into Gaussians. 
In the prediction step, 
assume that $\mu_n \approx \N(\mean_n, \Cov_n)$, then under~\cref{eq:evolve}, $\hat{\mu}_{n+1} = \N(\pmean_{n+1}, \pCov_{n+1})$ is also Gaussian and satisfies
\begin{equation}
\label{eq:KF_pred}
\begin{split}
&\pmean_{n+1} = \E[\theta_{n+1}|Y_n] =  \mean_n \qquad \pCov_{n+1} = \mathrm{Cov}[\theta_{n+1}|Y_n] = \Cov_{n} + \Sigma_{\omega}.
\end{split}
\end{equation}

In the analysis step, we assume that the joint distribution of     $\{\theta_{n+1}, y_{n+1}\} | Y_{n}$ can be approximated by a Gaussian distribution
\begin{equation}
\label{eq:KF_joint}
    \N\Big(
    \begin{bmatrix}
    \pmean_{n+1}\\
    \py_{n+1}
    \end{bmatrix}, 
    \begin{bmatrix}
  \pCov_{n+1} & \pCov_{n+1}^{\theta p}\\
    {{\pCov_{n+1}}^{\theta p}}{}^{T} & \pCov_{n+1}^{pp}
    \end{bmatrix}
    \Big),
\end{equation}
where 
\begin{equation}
\label{eq:KF_joint2}
\begin{split}
    &\py_{n+1} = \E[y_{n+1}|Y_n] = \E[\G(\theta_{n+1})|Y_n], \\
    &\pCov_{n+1}^{\theta p} = \mathrm{Cov}[\theta_{n+1}, y_{n+1}|Y_n] = \mathrm{Cov}[\theta_{n+1}, \G(\theta_{n+1})|Y_n], \\
    &\pCov_{n+1}^{p p} = \mathrm{Cov}[y_{n+1}|Y_n] = \mathrm{Cov}[\G(\theta_{n+1})|Y_n] + \Sigma_{\nu}.\\
\end{split}
\end{equation}
Conditioning the Gaussian in \cref{eq:KF_joint} to find $\theta_{n+1}|{Y_n, y_{n+1}} = \theta_{n+1}|Y_{n+1}$, gives the following expressions for the mean $\mean_{n+1}$ and covariance $\Cov_{n+1}$ of the approximation to $\mu_{n+1}$ : 
\begin{equation}
\label{eq:KF_analysis}
    \begin{split}
        \mean_{n+1} &= \pmean_{n+1} + \pCov_{n+1}^{\theta p} (\pCov_{n+1}^{p p})^{-1} (y_{n+1} - \py_{n+1}), \\
         \Cov_{n+1} &= \pCov_{n+1} - \pCov_{n+1}^{\theta p}(\pCov_{n+1}^{p p})^{-1} {\pCov_{n+1}^{\theta p}}{}^{T}.
    \end{split}
\end{equation}
\begin{remark}
Assume $\mu_{n}=\N(\mean_n,\Cov_n)$, 
for linear $\G$, the Gaussian approximation $\N(\mean_{n+1}, \Cov_{n+1})$ in~\cref{eq:KF_analysis} is exact, namely $\mu_{n+1}=\N(\mean_{n+1},\Cov_{n+1})$; for nonlinear but close to linear $\G$, the Gaussian approximation $\N(\mean_{n+1}, \Cov_{n+1})$ is a good approximation of the conditional distribution $\mu_{n+1}$, an upper bound of the Kullback–Leibler divergence between them is presented in \cref{sec:app:GA}.
\end{remark}
\Cref{eq:KF_pred,eq:KF_joint,eq:KF_joint2,eq:KF_analysis} establish a conceptual algorithm for application of Gaussian approximation to solve the inverse problems.
And the integrals appearing in~\cref{eq:KF_joint2} are approximated by the extended or unscented approach, which is detailed in the following subsections.

\subsection{Extended Kalman Inversion~(ExKI)}
ExKI approximates integrals in \cref{eq:KF_joint2} analytically by applying first-order Taylor expansion to $\G(\theta_{n+1})$ at the conditional expectation $\E[\theta_{n+1}|y] = \pmean_{n+1}$, 
\begin{equation}
\label{eq:ExKI_Taylor}
    \G(\theta_{n+1}) \approx \G(\pmean_{n+1}) + d\G(\pmean_{n+1}) (\theta_{n+1} - \pmean_{n+1}).
\end{equation}
And the iteration procedure of ExKI becomes:
\begin{itemize}
    \item Prediction step : 
    \begin{equation}
    \label{eq:ExKI-1.1}
    \begin{split}
        &\pmean_{n+1} = \mean_{n} \qquad \pCov_{n+1} = \Cov_{n} + \Sigma_{\omega}.\\
    \end{split}
    \end{equation}
    \item Analysis step :
    \begin{equation}
    \label{eq:ExKI-1.2}
    \begin{split}
    &\py_{n+1} = \G(\pmean_{n+1}), \\
    &\pCov_{n+1}^{\theta p} = \pCov_{n+1} d\G (\pmean_{n+1})^T,\\
    &\pCov_{n+1}^{p p} = d\G (\pmean_{n+1})\pCov_{n+1}d\G (\pmean_{n+1})^T + \Sigma_{\nu},\\
    &\mean_{n+1} = \pmean_{n+1} + \pCov_{n+1}^{\theta p}(\pCov_{n+1}^{pp})^{-1}(y - \py_{n+1}),\\
    &\Cov_{n+1} = \pCov_{n+1} - \pCov^{\theta p}_{n+1}(\pCov^{pp}_{n+1})^{-1}{\pCov^{\theta p}_{n+1}}{}^{T}.\\
    \end{split}
\end{equation}
\end{itemize}

\subsection{Unscented Kalman Inversion~(UKI)}
\label{subsec:UKI}
UKI approximates the integrals in \cref{eq:KF_joint2} by means of deterministic quadrature rules. This is the idea
of the unscented transform~\cite{julier1995new,wan2000unscented} which we now define.
\begin{definition}[Modified Unscented Transform~\cite{UKI}]
\label{def:unscented_transform}
Let denote Gaussian random variable $\theta \sim \N(\mean, \Cov) \in \R^{N_{\theta}}$, $2N_{\theta}+1$ symmetric sigma points are chosen deterministically:
\begin{equation}
    \theta^0 = \mean \qquad \theta^j = \mean + c_j [\sqrt{\Cov}]_j  \qquad \theta^{j+N_\theta}  = \mean - c_j [\sqrt{\Cov}]_j\quad (1\leq j\leq N_\theta),
\end{equation}
where $[\sqrt{\Cov}]_j$ is the $j$th column of the Cholesky factor of $\Cov$. The quadrature rule approximates the mean and covariance of the transformed variable $\G_i(\theta)$ as follows,  
\begin{equation}
\label{eq:ukf-original}
\begin{split}
    \E[\G_i(\theta)] \approx \G_i(\theta^0)\qquad 
    \textrm{Cov}[\G_1(\theta), \G_2(\theta]  \approx \sum_{j=1}^{2N_{\theta}} W_j^{c} (\G_1(\theta^j) - \E\G_1(\theta))(\G_2(\theta^j) - \E\G_2(\theta))^T. 
\end{split}
\end{equation}
Here these constant weights are 
\begin{equation}
\begin{split}
    &c_1=c_2\cdots=c_{N_\theta} = \sqrt{N_\theta +\lambda} \quad W_j^{c} =\frac{1}{2(N_\theta+\lambda)}~(j=1,\cdots,2N_{\theta})\\
    &\lambda = a^2 (N_\theta + \kappa) - N_\theta \qquad \kappa = 0 \qquad a=\min\{\sqrt{\frac{4}{N_\theta + \kappa}},  1\}.
\end{split}
\end{equation}
\end{definition}

Consider the Gaussian approximation algorithm defined by \cref{eq:KF_pred,eq:KF_joint,eq:KF_joint2,eq:KF_analysis}.
By utilizing the aforementioned quadrature rule, the iteration procedure of the UKI becomes:
\begin{itemize}
\item Prediction step : 
    \begin{align*}
        &\pmean_{n+1} = \mean_{n} \qquad \pCov_{n+1} = \Cov_n + \Sigma_{\omega}.\\
    \end{align*}
    \item Generate sigma points :
    \begin{align*}
    &\pp_{n+1}^0 = \pmean_{n+1}, \\
    &\pp_{n+1}^j = \pmean_{n+1} + c_j [\sqrt{\pCov_{n+1}}]_j \quad (1\leq j\leq N_\theta),\\ 
    &\pp_{n+1}^{j+N_\theta} = \pmean_{n+1} - c_j [\sqrt{\pCov_{n+1}}]_j\quad (1\leq j\leq N_\theta).
    \end{align*}
\item Analysis step :
  \begin{equation}
  \label{eq:UKI-analysis}
  \begin{split}
        &\ppy^j_{n+1} = \mathcal{G}(\pp^j_{n+1}) \qquad \py_{n+1} = \ppy^0_{n+1},\\
         &\pCov^{\theta p}_{n+1} = \sum_{j=1}^{2N_\theta}W_j^{c}
        (\pp^j_{n+1} - \pmean_{n+1} )(\ppy^j_{n+1} - \py_{n+1})^T, \\
        &\pCov^{pp}_{n+1} = \sum_{j=1}^{2N_\theta}W_j^{c}
        (\ppy^j_{n+1} - \py_{n+1} )(\ppy^j_{n+1} - \py_{n+1})^T + \Sigma_{\nu},\\
        &\mean_{n+1} = \pmean_{n+1} + \pCov^{\theta p}_{n+1}(\pCov^{pp}_{n+1})^{-1}(y - \py_{n+1}),\\
        &\Cov_{n+1} = \pCov_{n+1} - \pCov^{\theta p}_{n+1}(\pCov^{pp}_{n+1})^{-1}{\pCov^{\theta p}_{n+1}}{}^{T}.\\
    \end{split}
    \end{equation}
\end{itemize}

\subsection{Choice of Hyperparameters}
\label{ssub:hyperparameters}
The hyperparameters $\Sigma_{\omega}$ and $\Sigma_{\nu}$ in the stochastic dynamical system~\cref{eq:dynamics} are chosen at the $n$-th iteration, as following
\begin{equation}
\label{eq:hyperparameters}
\boxed{
    \Sigma_{\nu} = 2\Sigma_{\eta} \quad \textrm{ and } \quad \Sigma_{\omega} =  \Cov_n,
    }
\end{equation} 
It is worth mentioning the artificial evolution error covariance~$\Sigma_{\omega}$ is adaptively updated as the estimated covariance~$\Cov_{n}$. This choice of $\Sigma_{\omega}$ marks the key difference to the previous UKI discussed in~\cite{UKI, UKI2}.

\begin{remark}
A useful way to think of the procedure is through the analogy to Metropolis–Hastings algorithm (MH). 
The prediction step corresponds to the proposal step in MH, and hence, the adaptively updating of $\Sigma_{\omega}$ is similar to Adaptive Proposal and Adaptive Metropolis strategies~\cite{haario1999adaptive,haario2001adaptive}, 
where the proposal distribution is continuously adapted using the information contained in the sample path of the Markov chain. 
The analysis step corresponds to the acceptance-rejection step in MH.
\end{remark}

\section{Theoretical Insights}
\label{sec:Theorem}
The aforementioned Kalman inversion methodology recursively applies Gaussian approximation in each iteration to solve the inverse problem. A useful way to think of the iterative approach is through the following updating relation
\begin{equation*}
    \mu_{n+1}(d\theta) = \frac{1}{Z_n}\exp\big(-\Phi(\theta,y)\big) \mu_n(d\theta). 
\end{equation*}
In the limit of large $n$, $\mu_n$
will tend to concentrate on minimizers of $\Phi$; this follows from the identity
\begin{equation*}
\mu_n(d\theta)=\frac{1}{\bigl(\Pi_{\ell=0}^{n-1}Z_{\ell}\bigr)}\exp\bigl(-n\Phi(\theta,y)\bigr)\mu_0(d\theta).
\end{equation*}
The prior information fades away, and therefore, Kalman inversion methodology generally does not converge to the posterior distribution $\theta|y$.

However, in this section, we will show that Kalman inversion with the hyperparameters defined in~\cref{eq:hyperparameters}~ provides an accurate Gaussian approximation to the posterior distribution $\theta|y$ with an \textbf{uninformative prior} under certain conditions.

\subsection{The Linear Setting}
In this subsection, we study the UKI in the context of linear inverse problems, for which $\G(\cdot) = G$. 
Thanks to the linearity, equations~\eqref{eq:KF_joint2} are reduced to 
\begin{align*}
    \py_{n+1} = G\mean_n, \quad \pCov_{n+1}^{\theta p} = \pCov_{n+1} G^T,\quad  \textrm{and} \quad \pCov_{n+1}^{pp} = G  \pCov_{n+1} G^T + \Sigma_{\nu}.
\end{align*}
The update equations~\eqref{eq:KF_analysis} become
\begin{equation}
\label{eq:Lin_KF_analysis}
    \begin{split}
        \mean_{n+1} &= \mean_{n} + \pCov_{n+1} G^T (G  \pCov_{n+1} G^T + \Sigma_{\nu})^{-1} (y - G\mean_{n}), \\
         \Cov_{n+1}&= \pCov_{n+1} - \pCov_{n+1} G^T(G  \pCov_{n+1} G^T + \Sigma_{\nu})^{-1} G \pCov_{n+1},
    \end{split}
\end{equation}
with $\pCov_{n+1} =  \Cov_{n} + \Sigma_{\omega}$.
We have the following theorem about the convergence of the Kalman inversion:

\begin{theorem}
\label{th:lin_converge}
Assume the initial covariance matrix $\Cov_{0} \succ 0$ is strictly
positive definite. 
The iteration for the conditional mean $\mean_n$ and precision matrix $\Cov^{-1}_{n}$ characterizing the distribution of $\theta_n|Y_n$
converges exponentially fast to limit $\mean_{\infty}, \Cov^{-1}_{\infty}.$
Furthermore
the limiting mean $\mean_{\infty}$ is a minimizer
of the unregularized least squares
functional $\Phi$~\eqref{eq:KI2-b};
the limiting precision matrix $\Cov^{-1}_{\infty} = G^T\Sigma_{\eta}^{-1}G$, which is the 
posterior precision matrix with an uninformative prior.
\end{theorem}
\begin{proof}
The proof is in \cref{sec:app:Linear-UKI}.
\end{proof}

\begin{remark}
When $G$ has empty null space, which corresponds to a well-posed inverse problem, the posterior distribution with an uninformative prior exists. The covariance matrix $\{\Cov_n\}$ converges to $\Big(G^T\Sigma_{\eta}^{-1}G\Big)^{-1}$, which is the posterior covariance matrix with the uninformative prior.
\end{remark}

\begin{remark}
When $G$ has non-empty null space, namely $G^T\Sigma_{\eta}^{-1}G$ is singular, and hence the posterior distribution with an uninformative prior does not exist. $\Cov_{\infty}^{-1}$ is singular and therefore, the covariance matrix $\{\Cov_{n}\}$ diverges to $\infty$. We have the following bound
\begin{equation*}
    \Cov_n \preceq 2^n \Cov_0.
\end{equation*}
\end{remark}

\subsection{The Nonlinear Setting}
\label{subsec:nonlinear-the}
In this subsection, we study the UKI in the context of  nonlinear well-posed inverse problems.
The following pull-back distribution bridges the posterior distribution~\eqref{eq:KI3} and the stationary Gaussian distribution obtained by the Kalman inversion.
\begin{definition}[Pull-back random variable]
Assume $N_{\theta} = N_y$, given a bijective function $\G: R^{N_{\theta}}\mapsto R^{N_y}$ and an arbitrary vector $y\in R^{N_y}$. For any random  vector $\xi \in R^{N_y}$, the corresponding pull-back random variable is defined as
\begin{equation}
\label{eq:inv_distribution}
    \theta^{\xi} = \G^{-1}(y - \xi).
\end{equation}
\end{definition}

We have the following theorem about the posterior distribution $\theta | y$ with an uninformative prior defined in~\cref{eq:KI3} and the pull-back distribution of $\theta^{\eta}= \G^{-1}(y - \eta)$:
\begin{theorem}
\label{theorm:nonlinear}
Consider the posterior density function $p(\theta | y) $ with an uninformative prior
\begin{equation}
\label{eq:posterior_distribution}
    p(\theta | y) \sim  \frac{1}{Z} e^{- \Phi (\theta, y)},
\end{equation}
and the pull-back density function $p(\theta^{\eta})$. 
We assume 
\begin{enumerate}
    \item the map $\G$ is a bijection, hence, $N_\theta = N_y$ and $\G^{-1}$ exists,
    \item $\Big| \texttt{det}\frac{d\G^{-1}(\theta)}{d \theta} \Big|$ has Lipschitz property: 
    \begin{equation}
        \Big|| \texttt{det}\frac{d\G^{-1}(\theta_1)}{d \theta}| - | \texttt{det}\frac{d\G^{-1}(\theta_2)}{d \theta}| \Big| \leq c_0\lVert\theta_1 -  \theta_2\rVert^{c_1},
    \end{equation}
    \item $\G^{-1}(\theta)$ does not grow too fast, and we have 
    \begin{equation*}
        \frac{1}{\sqrt{(2\pi)^{N_y} \texttt{det}\Sigma_{\eta}}}\int  e^{- \frac{1}{2}(y - \theta)^T \Sigma^{-1}_{\eta} (y - \theta)} \lVert \G^{-1}(\theta) \rVert^4 d\theta \leq c_2, 
    \end{equation*}
    \item The normalization constant
    \begin{equation*}
    \begin{split}
        Z = \int  e^{- \Phi (\theta, y)} d\theta  = \int  e^{- \frac{1}{2}(y - \theta^{'})^T \Sigma^{-1}_{\eta} (y - \theta^{'})} \Big| \texttt{det}\frac{d\G^{-1}(\theta^{'})}{d\theta}\Big| d\theta' \textrm{ where } \theta' = \G(\theta)\\
    \end{split}
    \end{equation*}
    exists and has positive lower bound $Z \geq c_3 > 0$,
    \item  $c_i$ and $N_\theta$ are $\bigO(1)$ constants, and the spectral radius of $\Sigma_{\eta}$, $\rho(\Sigma_{\eta})$ is small enough.
\end{enumerate}
Then  $p(\theta | y) $ and $p(\theta^{\eta})$
have close mean and covariance, which satisfy
\begin{equation}
\label{eq:post-pull-bound}
    \rVert \mean - \mean^{\eta}  \rVert_{\infty} = \bigO(\rho(\Sigma_{\eta})^{c_1} \sqrt{\texttt{det}\Sigma_{\eta}}) \qquad \textrm{and} \qquad
    \rVert \Cov - \Cov^{\eta}  \rVert_{\infty} = \bigO(\rho(\Sigma_{\eta})^{c_1}\sqrt{\texttt{det}\Sigma_{\eta}}),
\end{equation}
here $\mean$ and $\Cov$ and $\mean^{\eta}$ and $\Cov^{\eta}$ are the mean and covariance of $p(\theta | y) $  and  $p(\theta^{\eta})$, respectively.
\end{theorem}
\begin{proof}
The proof is in \cref{sec:app:theorm:nonlinear:proof}.
\end{proof}

We have the following theorem about the pull back distribution of $\theta^{\eta}$ and the ExKI,
\begin{theorem}
\label{th:ExKI}
Assume $N_{\theta} = N_y$ and $\G$ is a bijection. Any stationery mean and covariance $\mean_{*}, \Cov_{*}$ obtained by the ExKI, which satisfy that both $d\G(m_{*})$ and $\Cov_{*}$ are non-singular, then 
\begin{equation}
\label{eq:ExKI-stationery1}
    \mean_{*} = \G^{-1}(y)  \quad \textrm{and} \quad 
    \Cov_*^{-1} = d\G(\mean_{*})^T\Sigma_{\eta}^{-1}d\G(\mean_{*}).
\end{equation}
And they are the mean and covariance estimation of the pull-back random variable $\theta^{\eta}$ obtained by the extended Kalman filter.
\end{theorem}
\begin{proof}
The proof is in \cref{sec:app:ExKI:proof}.
\end{proof}

Combining~\cref{theorm:nonlinear,th:ExKI} , we have, in the presence of small observation error, the ExKI is able to well approximate the mean and covariance of the posterior distribution with an uninformative prior under certain conditions on the operator $\G$. 
The UKI~\cite{UKI} further applies averaging on the inverse function $\G$ and its gradient $d\G$, which leads to $\F_u\G$ and $\F_u d \G$. 
By the analogy to the ExKI, the UKI is also able to approximate the mean and covariance of the posterior distribution with an uninformative prior under similar conditions.

\begin{remark}
Numerical studies presented in~\cref{ssec:app:darcy,ssec:app:NS} indicate the UKI delivers good Gaussian approximation to the posterior distribution also for over-determined inverse problems.
\end{remark}

\begin{remark}
Numerical studies presented in~\cref{ssec:app:1-p} indicate the injection and the Lipschitz continuity of $\G$ are necessary for accurate posterior mean and covariance approximation.
\end{remark}

\section{Applications}
\label{sec:app}
In this section, 
we present numerical study of the UKI equipping with the stochastic dynamics~\cref{eq:hyperparameters} for approximating the posterior distribution.
To be concrete, we initialize UKI with $\theta_0 \sim \N(\mean_0, \Cov_0)$.
The theoretical results in~\cref{sec:Theorem}  indicate that the converged mean and covariance are independent of the initial guess. 
However, the UKI converges faster, when the initial guess is closer to the posterior mean and  covariance. 
We set $\mean_0 $ to be the prior mean $r_0$. Since the prior covariance is generally very large, the initial covariance $C_0$ is not necessarily the prior covariance.
Specific choices of $r_0$ and $\Cov_0$ will differ between examples and will be spelled out in each example. 
The reference posterior distributions are computed with the Markov Chain Monte Carlo method~(MCMC), specifically the random walk Metropolis algorithm~\cite{gelman1997weak}, with sufficient iterations.
For comparison, we also apply the affine invariant Markov Chain Monte Carlo ensemble sampler~\cite{goodman2010ensemble,foreman2013emcee} (emcee),  
the Sequential Monte Carlo method~\cite{del2006sequential,kantas2014sequential,beskos2015sequential}~(SMC), and ensemble transform Kalman inversion~\cite{bishop2001adaptive,wang2003comparison,UKI2}~(ETKI).

\begin{itemize}
    \item Nonlinear 1-parameter model problems: the behavior of the UKI is studied on different forward maps, including discontinuous functions, non-injective functions, and etc.
    \item Nonlinear 2-parameter model problem: this is a counterexample against the ensemble Kalman filter~\cite{ernst2015analysis,garbuno2020interacting}. This problem demonstrates the effectiveness of the modified stochastic dynamical system, which enables the UKI to obtain accurate posterior approximation. The comparison with emcee, SMC, and ETKI is also presented.
    \item Nonlinear high dimensional model problem: this is a well-studied Darcy flow inverse problem. The comparison between UKI and MCMC is presented, good agreement in terms of mean and covariance estimations is achieved.
    \item Navier-Stokes problem: this is a model data assimilation problem in oceanography and meteorology, where the initial condition is recovered from noisy observations of the vorticity field at later times.
\end{itemize}

The code is accessible online:
\begin{center}
  \url{https://github.com/Zhengyu-Huang/InverseProblems.jl}
\end{center}

\subsection{Nonlinear 1-Parameter Model Problem}
\label{ssec:app:1-p}
The performance of the UKI is studied  numerically on the following nonlinear 1-parameter problems:
\begin{itemize}
    \item Exponential problem:
     \begin{equation*}
    y = \G(\theta) + \eta \qquad \G(\theta) = \exp\Big(\frac{\theta}{10}\Big).
    \end{equation*}
    \item Quadratic multimodal problem:
    \begin{equation*}
    y = \G(\theta) + \eta \qquad \G(\theta) = \theta^2. 
    \end{equation*}
    \item Cubic problem:
     \begin{equation*}
    y = \G(\theta) + \eta \qquad \G(\theta) = \theta^3. 
    \end{equation*}
    \item Sign discontinuous problem:
     \begin{equation*}
    y = \G(\theta) + \eta \qquad \G(\theta) = \textrm{sign}(\theta) + \theta^3. 
    \end{equation*}
    \item Hyperbola discontinuous problem:
     \begin{equation*}
    y = \G(\theta) + \eta \qquad \G(\theta) = \frac{1}{\theta}.
    \end{equation*}
\end{itemize}
We assume the observation is generated as $y = \G(\theta_{ref})$, where the reference solution $\theta_{ref} = 2$.
And the observation error is $\eta \sim \N(0, 0.1^2)$. For Bayesian inverse problems, we assume the prior distribution is $\N(1, 10^2)$. It is worth noticing, for these inverse problems, the observation error is small, and the prior is almost uninformative. However, $\G$ is not injective for the quadratic multimodal problem, and the Lipschitz property does not hold for these discontinuous problems.

The reference posterior distribution is approximated by the MCMC method with a step size $1.0$ and $5\times10^6$ samples (with a $10^6$ sample burn-in period). 
As for the UKI, $2$ initial conditions are considered, which are $\theta_0\sim\N(-1, 0.5^2)$ and $\theta_0\sim\N(1, 0.5^2)$. We find that only the hyperbola discontinuous problem is sensitive to the initial covariance for the UKI. 

The approximated posterior distributions are presented in \cref{fig:1-parameter}. Here the UKI results are from the 20th iteration. For the quadratic multimodal problem, the posterior distribution is a multimodal distribution, therefore, the UKI can only capture one modal, that is close to the initial condition $\theta_0$. For the hyperbola discontinuous problem, the UKI initialized on the different branch from $\theta_{ref}$ diverges to $-\infty$; but the UKI initialized on the same branch from $\theta_{ref}$ converges to $\theta_{ref}$ with good covariance estimation. This reveals the gradient-based nature of UKI, in contrast to the sampling-based nature of MCMC. For other cases, the distributions obtained by the UKI match well with the distributions delivered by the MCMC. It is worth mentioning, the UKI requires only 20 iterations~(60 forward solver evaluations), which is much cheaper than the MCMC method. 

\begin{figure}
     \centering
     \begin{subfigure}[b]{0.325\textwidth}
         \centering
         \includegraphics[width=\textwidth]{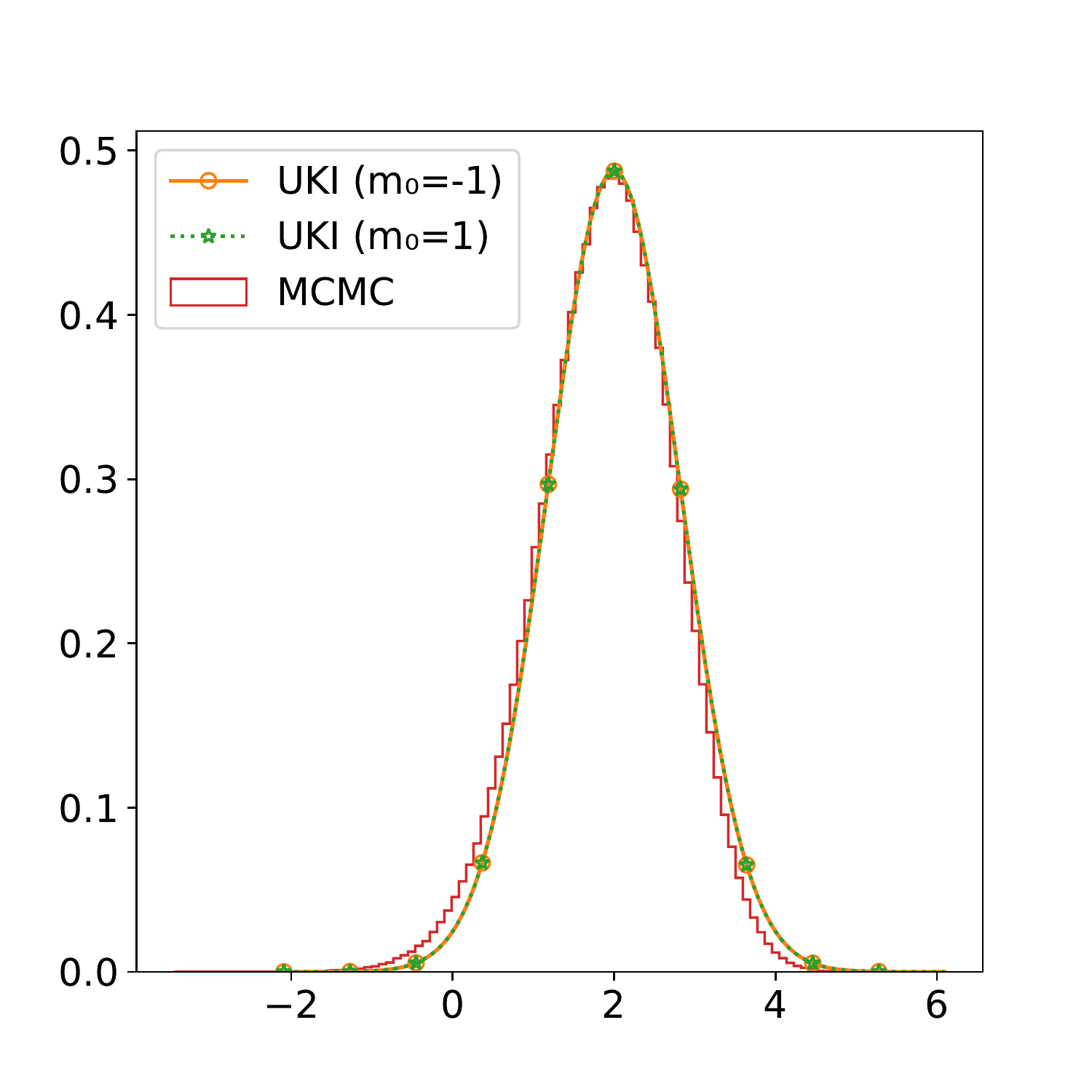}
         \caption{$\G(\theta)= \exp\Big(\frac{\theta}{10}\Big)$}
     \end{subfigure}
     \begin{subfigure}[b]{0.325\textwidth}
         \centering
         \includegraphics[width=\textwidth]{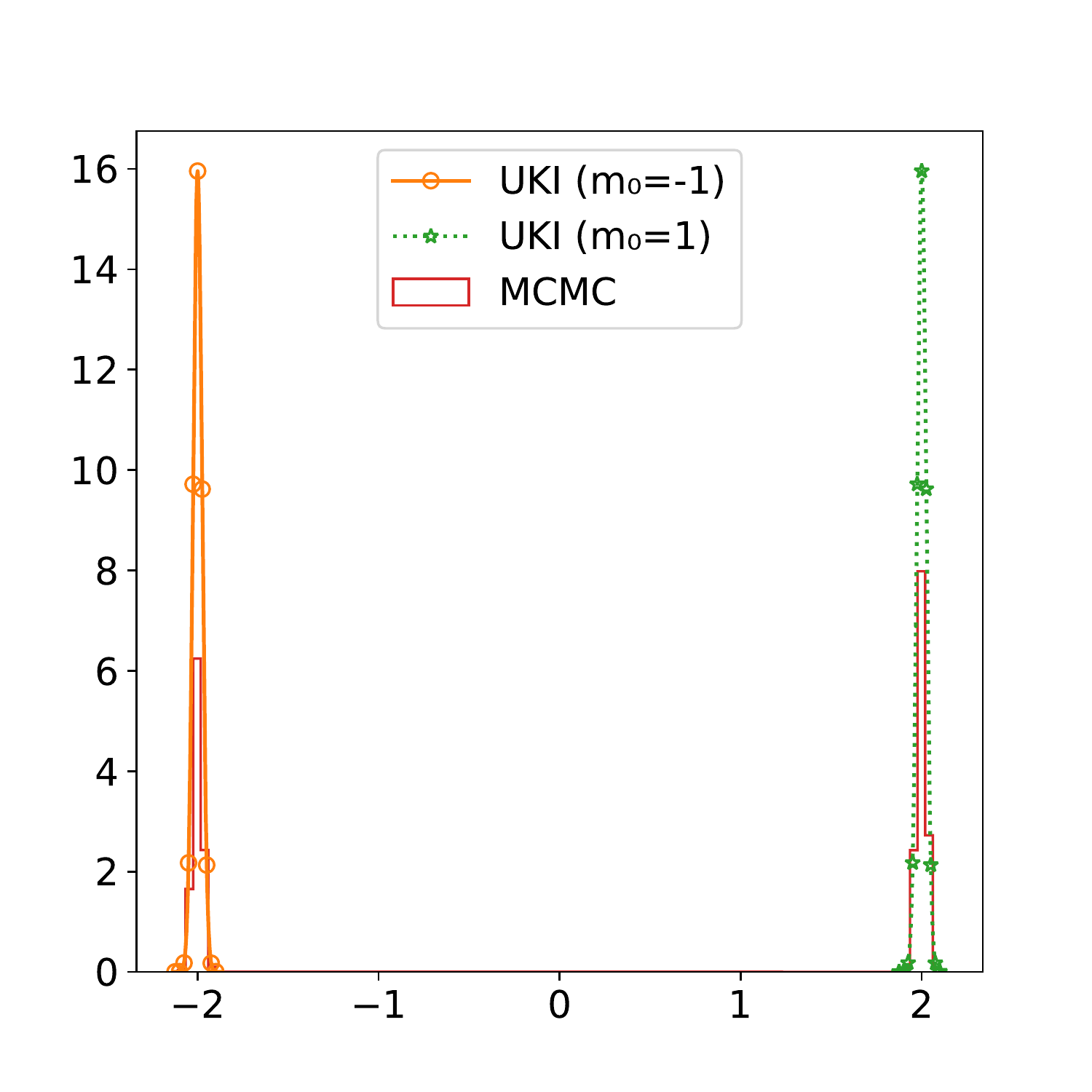}
         \caption{$\G(\theta) = \theta^2 $}
     \end{subfigure}
     \begin{subfigure}[b]{0.325\textwidth}
         \centering
         \includegraphics[width=\textwidth]{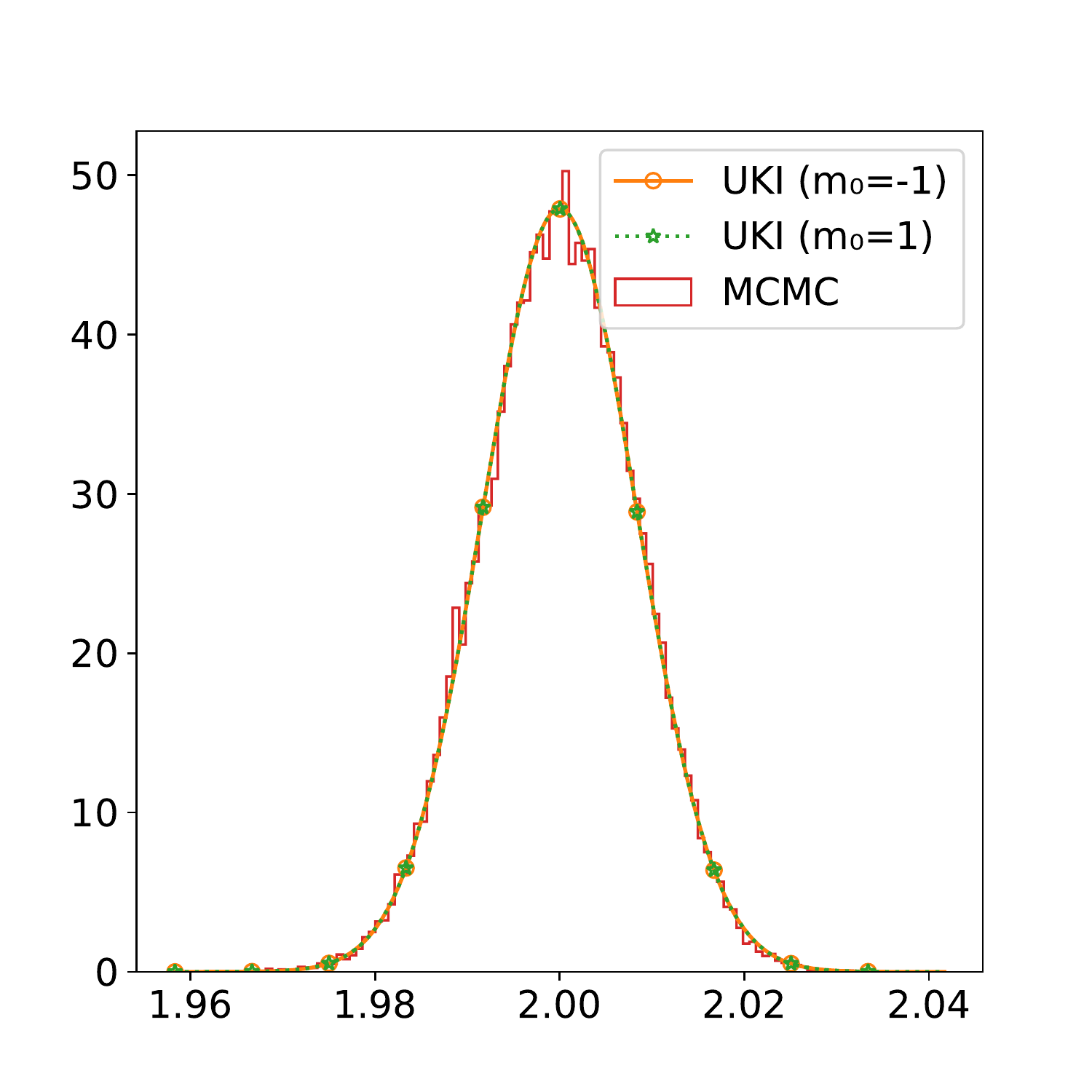}
         \caption{$\G(\theta) = \theta^3 $}
     \end{subfigure}
      \hfill
     \begin{subfigure}[b]{0.325\textwidth}
         \centering
         \includegraphics[width=\textwidth]{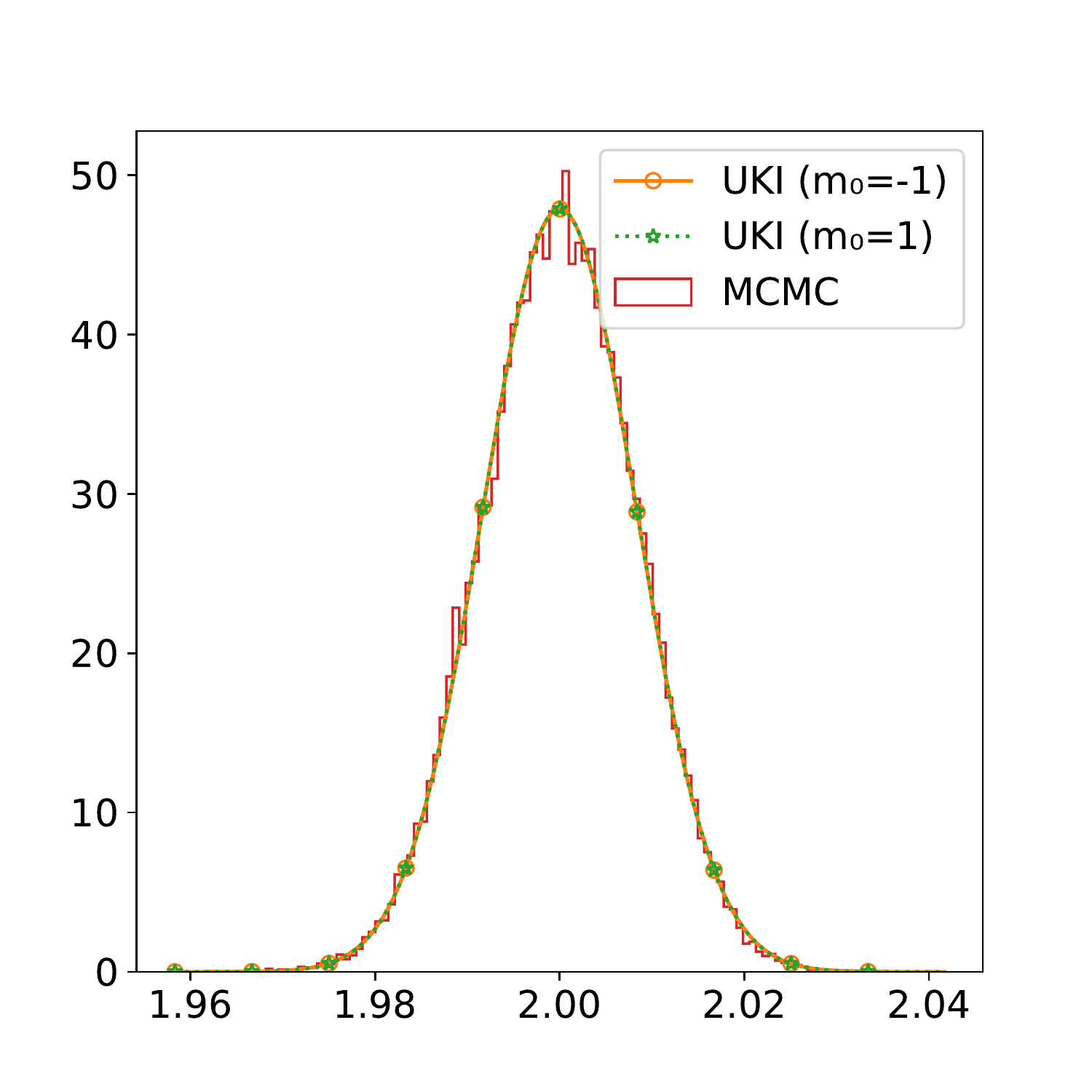}
         \caption{$\G(\theta) = \textrm{sign}(\theta) + \theta^3 $}
     \end{subfigure}
     \begin{subfigure}[b]{0.325\textwidth}
         \centering
         \includegraphics[width=\textwidth]{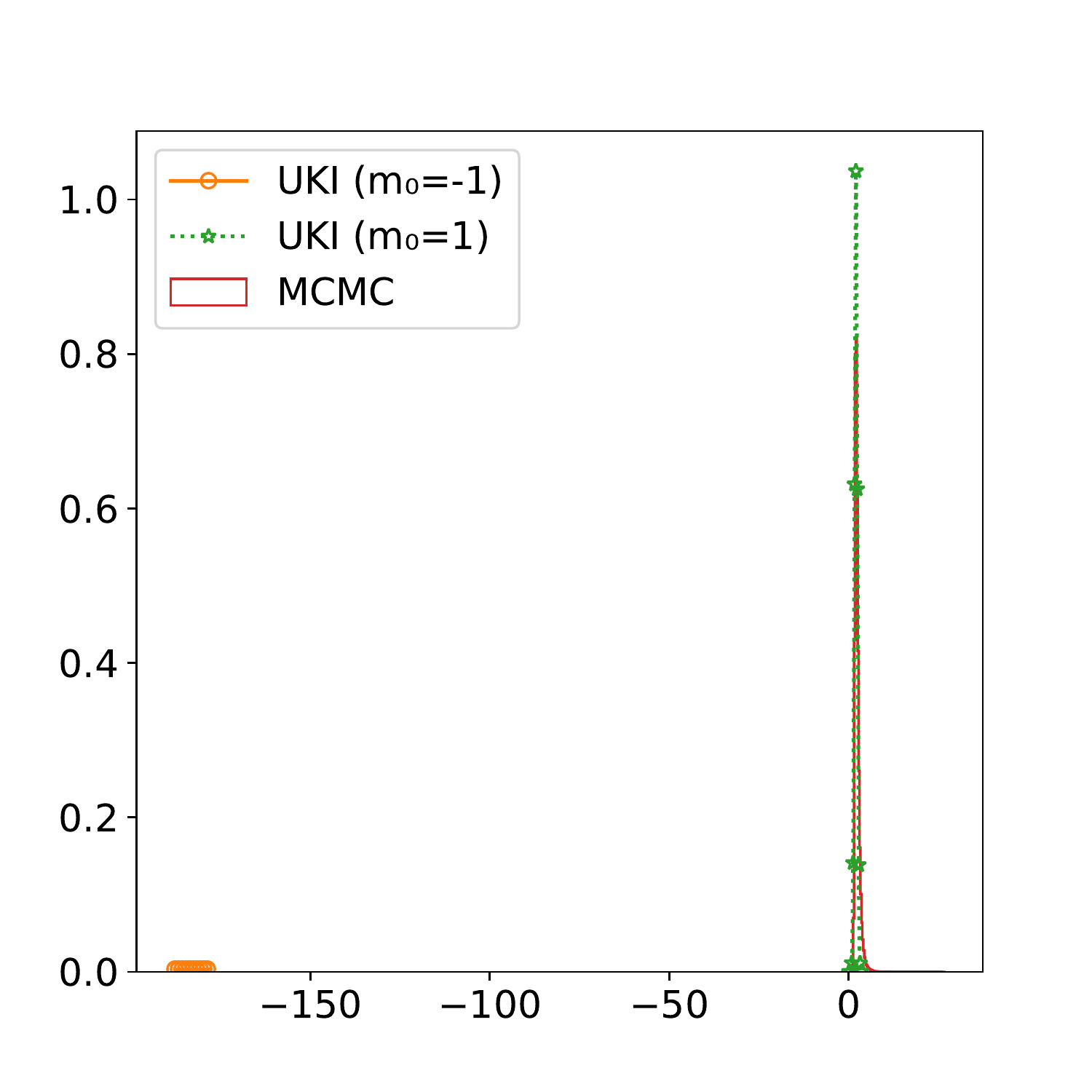}
         \caption{$\G(\theta) = \frac{1}{\theta}$}
     \end{subfigure}
     \begin{subfigure}[b]{0.325\textwidth}
         \centering
         \includegraphics[width=\textwidth]{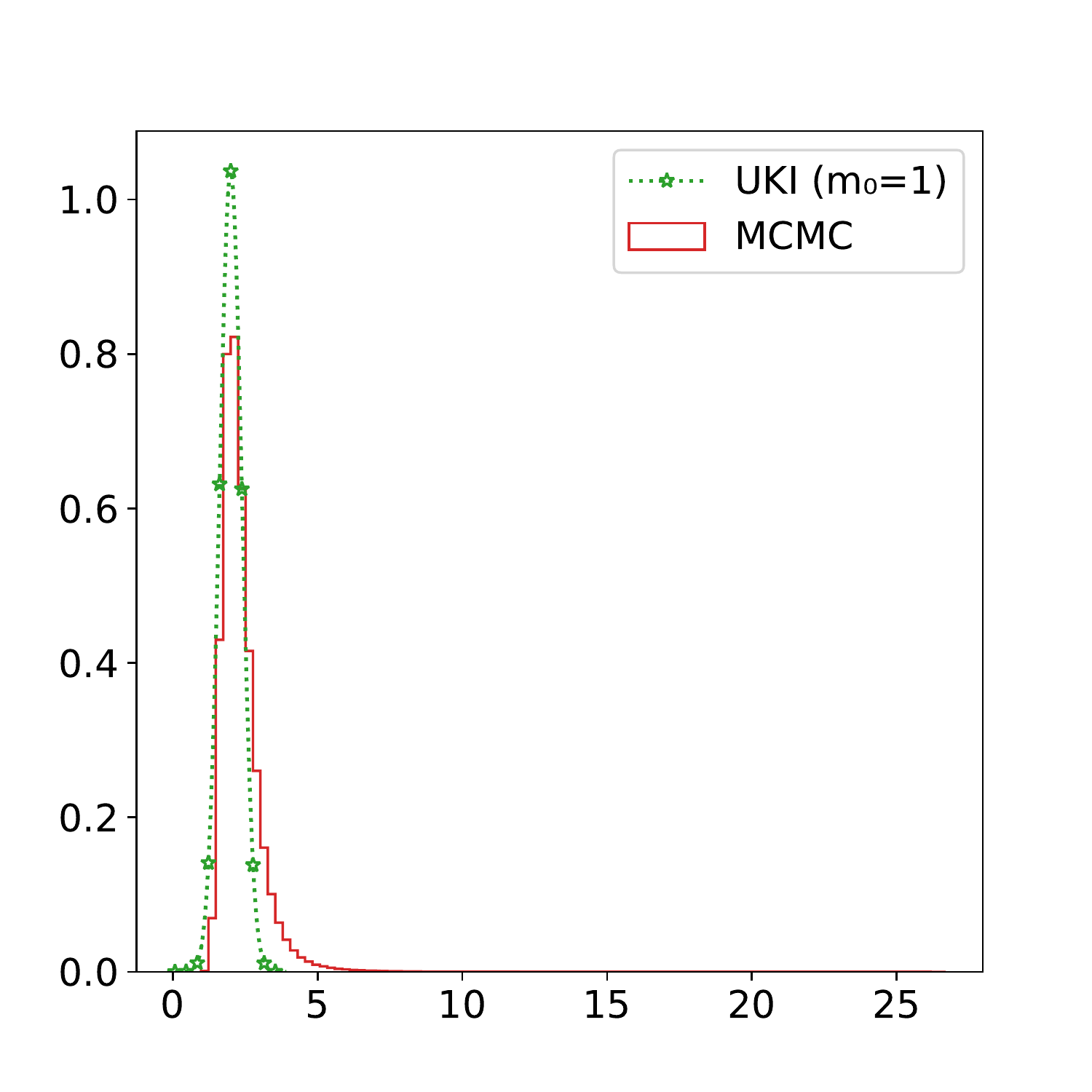}
         \caption{$\G(\theta) = \frac{1}{\theta}$ (Zoomed view)}
     \end{subfigure}
     \hfill
        \caption{Estimated posterior distributions by different approaches for different nonlinear 1-parameter model problems.}
        \label{fig:1-parameter}
\end{figure}

\subsection{Nonlinear 2-Parameter Model Problem}
\label{ssec:app:2-p}
The numerical experiment considered here is a counterexample against the ensemble Kalman filter, which is proposed in~\cite{ernst2015analysis} and further used in \cite{herty2018kinetic,garbuno2020interacting}. Consider the one-dimensional elliptic boundary-value problem
\begin{equation}
    -\frac{d}{dx}\Big(\exp(\theta_{(1)}) \frac{d}{dx}p(x)\Big) = 1, \qquad x\in[0,1]
\end{equation}
with boundary conditions $p(0) = 0$ and $p(1) = \theta_{(2)}$. The solution for this problem is given by  
\begin{equation*}
    p(x) = \theta_{(2)} x + \exp(-\theta_{(1)})\Big(-\frac{x^2}{2} + \frac{x}{2}\Big).
\end{equation*}
The inverse problem is to solve for $\theta = (\theta_{(1)},\, \theta_{(2)})^T$ with the observations $y = (p(x_1),\,p(x_2))^T$ at $x_1=0.25$ and $x_2=0.75$.
The Bayesian inverse problem is formulated as 
\begin{equation*}
    y = \G(\theta) + \eta \qquad \textrm{and} \qquad 
    \G(\theta) = \begin{bmatrix}
    p(x_1, \theta)\\
    p(x_2, \theta)
    \end{bmatrix},
\end{equation*}
here $\G(\theta)$ is the forward model operator.   The observation is $y=(27.5, 79.7)^T$ with observation error $\eta\sim\N(0, 0.1^2\I)$. And the prior is $\N(0,  \texttt{diag}\{1, 10^2\})$.

The reference posterior distribution is approximated by the MCMC method with a step size $1.0$ and  $5\times10^6$ samples (with a $10^6$ sample burn-in period).
For the UKI, the initial condition is $\theta_0 \sim \N(0, \texttt{diag}\{1, 10^2\})$.
The posterior distributions obtained by the UKI at the 5th, 10th, and 15th iterations are depicted in~\cref{fig:elliptic_cov1}-top. The mean converges efficiently to the true value $\theta_{ref}$ and the covariance obtained by the UKI matches well with that obtained by MCMC.  
For comparison, we also apply the UKI without updating the evolution error covariance $\Sigma_{\omega}$, and therefore  $\Sigma_{\omega}= \Cov_0$. The posterior distributions obtained at the 5th, 10th, and 15th iterations  are depicted in~\cref{fig:elliptic_cov1}-bottom.  The converged mean estimation matches well with $\theta_{ref}$, but the converged covariance is not accurate. This highlights the significance of adaptively updating $\Sigma_{\omega}$.

\begin{figure}[ht]
\centering
     \includegraphics[width=0.8\textwidth]{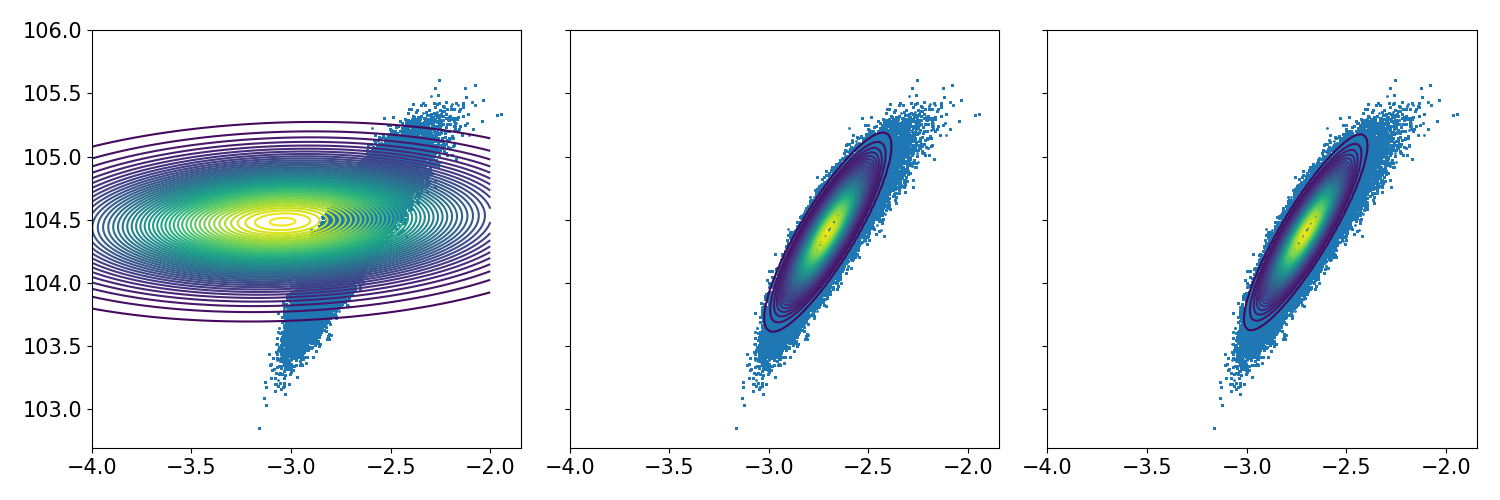}
     \includegraphics[width=0.8\textwidth]{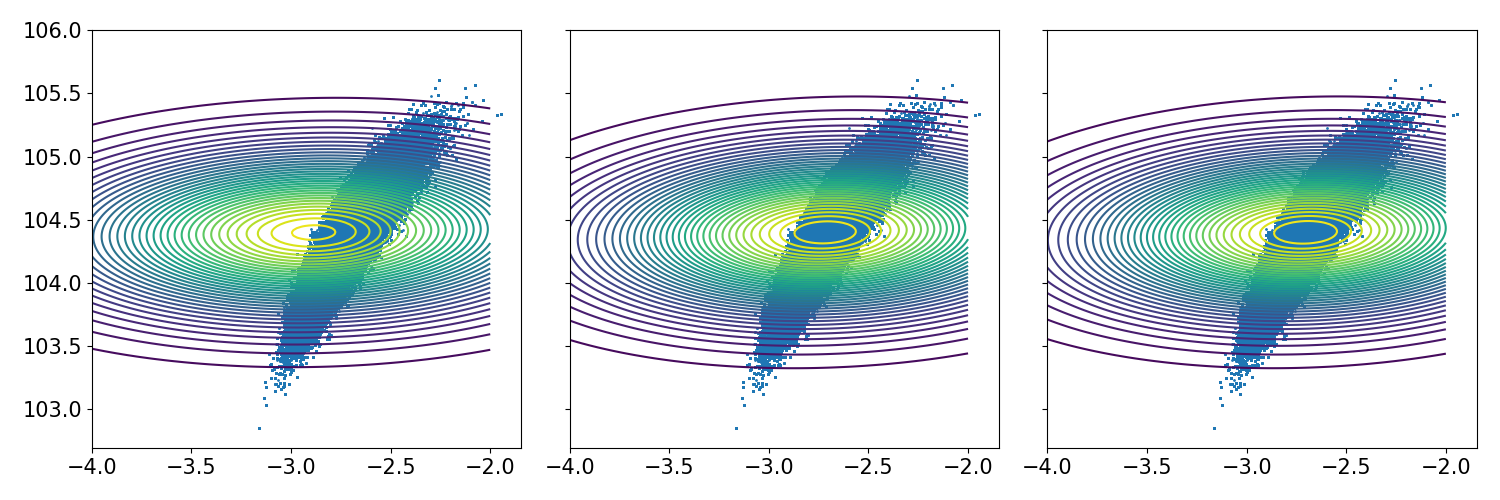}
    \caption{Contour plot: posterior distributions obtained by UKI~with (top)/without (bottom) updating the evolution error covariance $\Sigma_{\omega}$, at 5th, 10th, and 15th iterations~(from left to right); blue dots: reference posterior distribution obtained by MCMC for the 2-parameter model problem. x-axis is for $\theta_{(1)}$ and y-axis is for $\theta_{(2)}$.}
    \label{fig:elliptic_cov1}
\end{figure}

Moreover, we also report the  behaviors of emcee, SMC and ETKI on this problem.  They are all ensemble-based approaches. The ensemble size is set to be $J=100$, and the initial ensemble is drawn from  $\N(0, \texttt{diag}\{1, 10^2\})$.
For the emcee, the number of subensemble is 2 and the stretch move is used.  The estimated posterior distributions obtained with 100 iterations and 500 iterations are depicted in \cref{fig:elliptic_comp}~left.  
For the SMC, we choose the resampling threshold $M_{\textrm{thresh}} = 60$, the MCMC mover is the  random walk Metropolis algorithm with a step size $1.0$. The sequence of "bridging" densities, which enable us to connect the prior distribution to the posterior distribution, is $\mu_{n} \propto \exp\big(-\frac{n}{N} \Phi(\theta, y)\big) \mu_{0}$ with $n = 0,\, 1,\, \cdots,\, N$. The estimated posterior distributions with $N = 100$ and $N = 500$ are depicted in \cref{fig:elliptic_comp}~middle.
Both Monte Carlo samplers require more iterations and larger ensemble size compared with the number of $\sigma$-points in the UKI ($2N_{\theta} + 1$). However, it is worth mentioning these approaches are able to sample arbitrary distributions, the UKI is specifically designed to estimate posterior mean and covariance from the Bayesian measurement error model~\eqref{eq:KI} with small observation errors.
For the ETKI, it suffers divergence for the proposed stochastic dynamic system~\eqref{eq:dynamics}\eqref{eq:hyperparameters}. We revert to the widely used dynamic system~\cite{gu2006ensemble,iglesias2013ensemble, schillings2017analysis,schillings2018convergence}, with $\Sigma_{\omega} = 0$ and $\Sigma_{\nu} = \Sigma_{\eta} $. The estimated  posterior distributions obtained at the 1st iteration and the 30th iteration are depicted in~\cref{fig:elliptic_comp}. The ETKI suffers ensemble collapse, and hence, fails to deliver meaningful uncertainty information. A similar phenomenon~\cite{schillings2017analysis,schillings2018convergence,herty2018kinetic,garbuno2020interacting} has been reported  and discussed for its stochastic variant---Ensemble Kalman inversion.

\begin{figure}[ht]
\centering
    \includegraphics[width=0.8\textwidth]{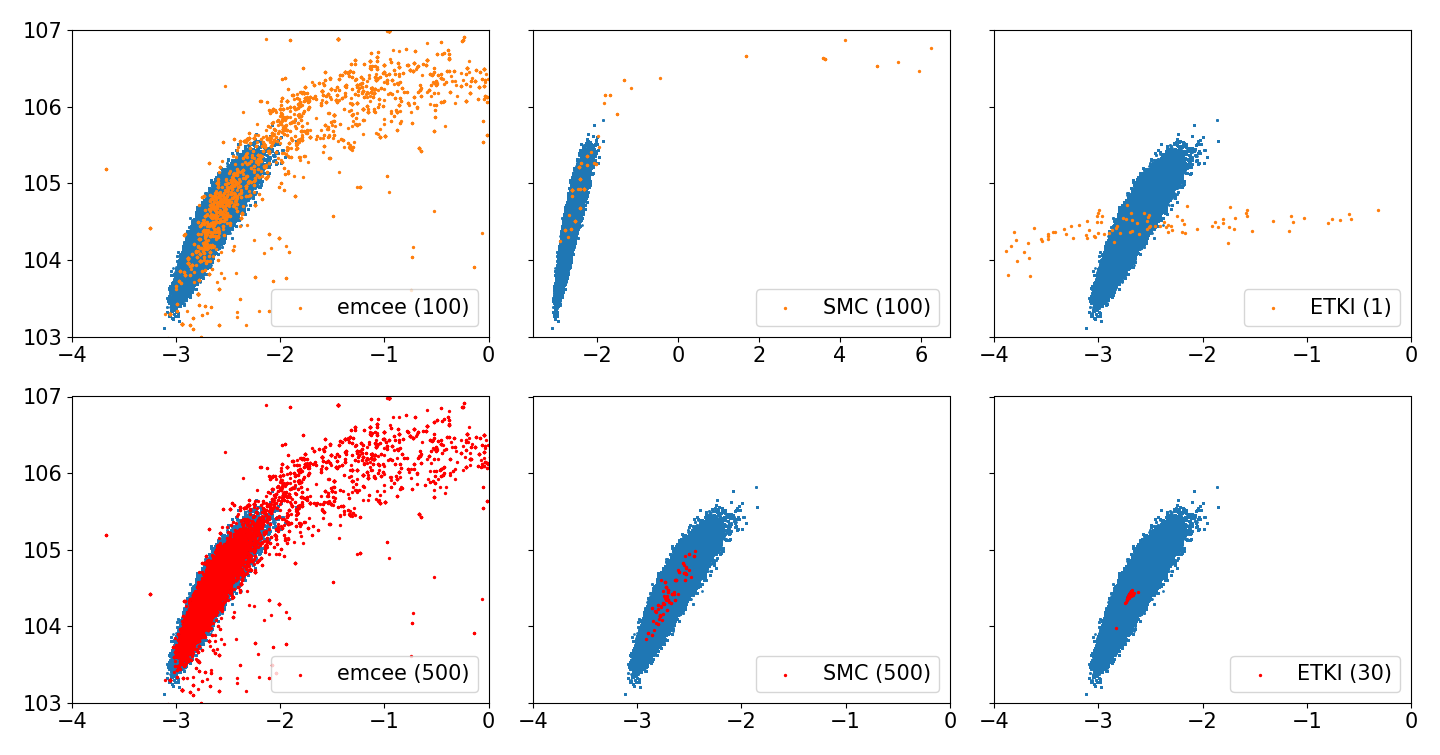}
    \caption{Estimated posterior distributions obtained by emcee with 100 iterations and 500 iterations (left); SMC with 100 iterations and 500 iterations (middle); ETKI at the 1st iteration and the 30th iteration (right) for the 2-parameter model problem. Blue dots represent the reference posterior distribution obtained by MCMC. x-axis is for $\theta_{(1)}$ and y-axis is for $\theta_{(2)}$.}
   \label{fig:elliptic_comp}
\end{figure}

\subsection{Nonlinear High-Dimensional Model Problem}
\label{ssec:app:darcy}
A similar  one-dimensional elliptic boundary-value problem but with high-dimensional parameters is considered in this section.
The equation describes the pressure field $p(x)$ in a porous medium defined by a positive random permeability field $a(x,\theta)$:
\begin{equation}
    -\frac{d}{dx}\Big(a(x, \theta) \frac{d}{dx}p(x)\Big) = f(x), \qquad x\in[0,1].
\end{equation}
Here  Dirichlet boundary conditions on the pressure are applied with $p(0) = 0$ and $p(1) = 0$, and $f$ defines the source of the fluid:
\begin{align*}
    f(x) = \begin{cases}
              1000 & 0 \leq x \leq \frac{1}{2}\\
              2000 & \frac{1}{2} < x \leq 1\\
            \end{cases}. 
\end{align*}

The random log-permeability field $\log a(x, \theta)$, depending on parameters $\theta\in\R^{N_\theta}$, is modeled 
as a log-Gaussian field with covariance operator
$$\mathsf{C} = (-\Delta + \tau^2 )^{-d},$$
where $-\Delta$ denotes the Laplacian on $D$ subject to homogeneous Neumann boundary conditions on the space of spatial-mean zero functions, 
$\tau > 0$ denotes the inverse length scale of the random field and $d  > 0$ determines its regularity ($\tau=3$ and $d=1$ in the present study\footnote{$d=1$ ensures the eigenvalues do not decay too fast, and hence all parameters are effective.}).

The log-Gaussian field is approximated by the following Karhunen-Lo\`{e}ve~(KL) expansion
\begin{equation}
\label{eq:KL-2d}
    \log a(x,\theta) = \sum_{l=1}^{+\infty} \theta_{(l)}\sqrt{\lambda_l} \psi_l(x),
\end{equation}
and the eigenpairs are of the form
\begin{equation*}
    \psi_l(x) = \sqrt{2}\cos(\pi l x) \quad \textrm{ and }
                 \quad \lambda_l = (\pi^2 l^2 + \tau^2)^{-d},
\end{equation*}
and $\theta_{(l)} \sim \N(0,1)$ i.i.d.  
The forward problem is solved by the finite difference method with $512$ grid points.

For the inverse problem, we generate a truth random field $\log a(x,\theta_{ref})$ with $N_{\theta} = 32$ and $\theta_{ref} \sim \N(0, \I^{32})$, which consists of the first $32$ KL modes~(See~\cref{fig:darcy-1d-ref}-left). 
The observation $y_{ref}$ consists of pointwise measurements of the pressure value $p(x)$ at $63$ equidistant points in the domain~(See~\cref{fig:darcy-1d-ref}-right), with the observation error $\eta \sim \N(0, 0.1^2\I)$. The prior is $\N(0, 10^2\I)$, where the covariance is large, and therefore the prior is almost uninformative.

\begin{figure}[ht]
\centering
    \includegraphics[width=0.48\textwidth]{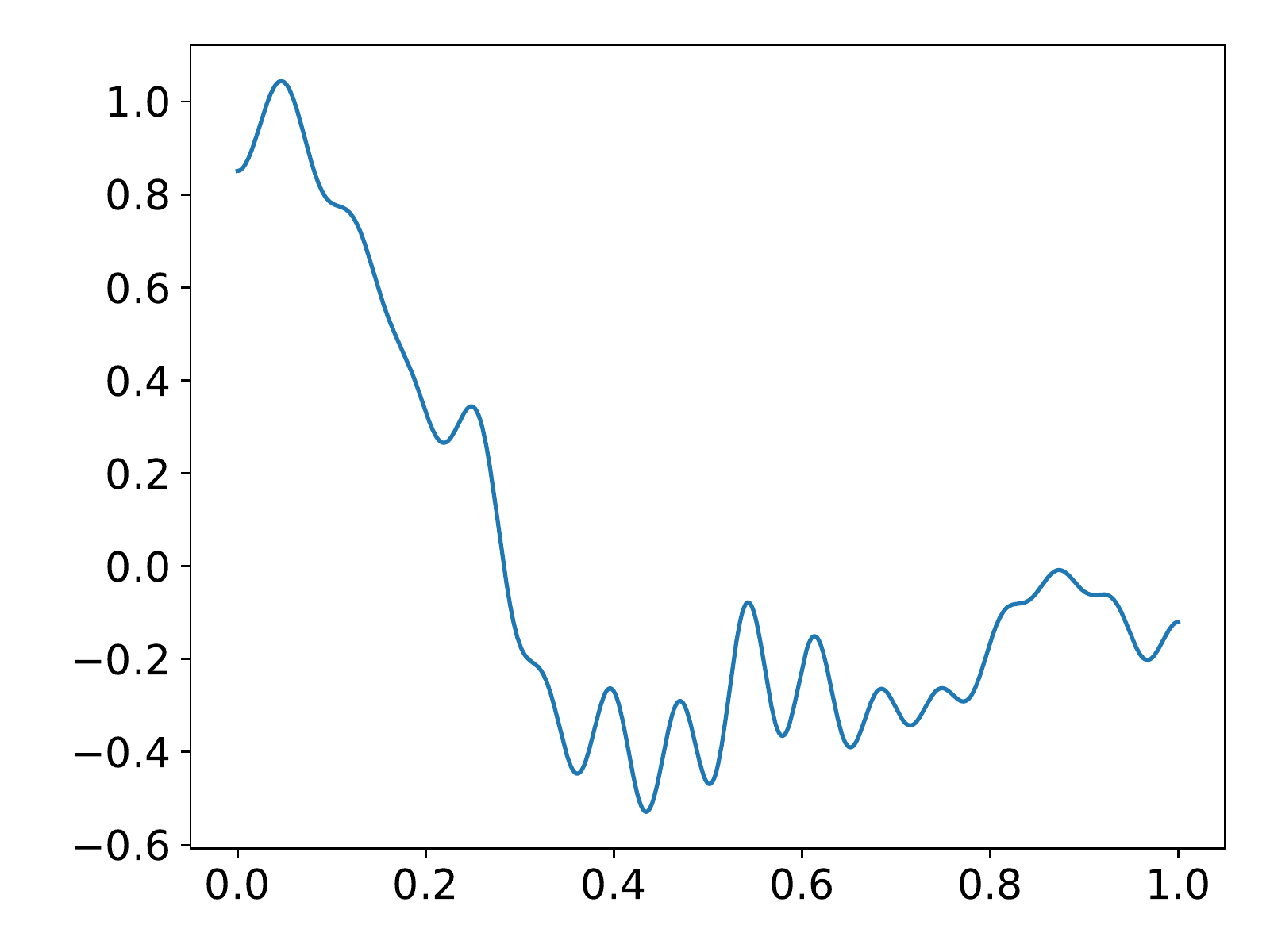}
    \includegraphics[width=0.48\textwidth]{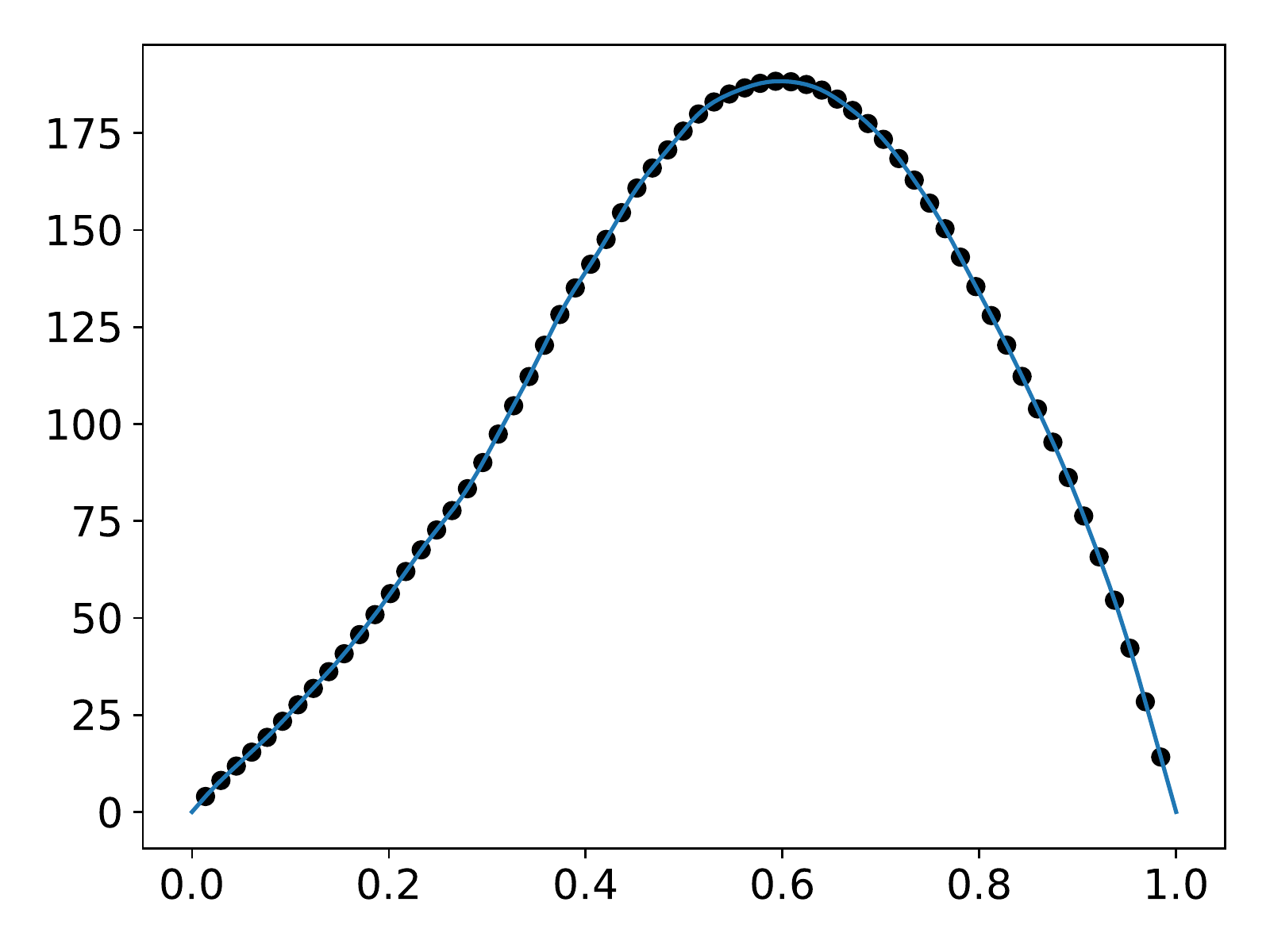}
    \caption{The reference log-permeability field $\log a(x,\theta_{ref})$~(left) and the pressure field with $63$ equidistant pointwise measurements~(right) of the Darcy flow problem.}
    \label{fig:darcy-1d-ref}
\end{figure}

The reference posterior distribution is approximated by MCMC with $4\times 10^8$ samples~(with a $2\times10^8$ sample burn-in period). In order to accelerate the convergence of MCMC, we initialize the sampling with $\theta_{ref}$ and choose a step size $10^{-2}$. For the UKI, the initial condition is $\theta_0 \sim \N(0, \I)$.

The estimated KL expansion parameters $\{\theta_{(i)}\}$ for the log-permeability field and the associated  3-$\sigma$ confidence intervals  obtained by the UKI at the 20th iteration and MCMC are depicted in~\cref{fig:darcy-1d-uq1}. Both UKI and MCMC converge to the true value $\theta_{ref}$~(the relative $L_2$ error $\displaystyle \frac{\lVert m_n - \theta_{ref}\rVert_2}{\lVert \theta_{ref}\rVert_2}$ obtained by UKI at the 20 iteration is about $10^{-2}$),
and the covariance estimations for each parameter match well with each other. 
The covariance for each pair ($\theta_{(i)},  \theta_{(j)}$) with $i \leq j$ obtained by UKI and MCMC are depicted in~\cref{fig:darcy-1d-uq2}. These pairs are sorted in an ascent order by comparing the first element and then the second. The UKI and MCMC deliver very similar covariance estimations, which conforms to the theoretical results in~\cref{subsec:nonlinear-the}.

\begin{figure}[ht]
\centering
    \includegraphics[width=0.98\textwidth]{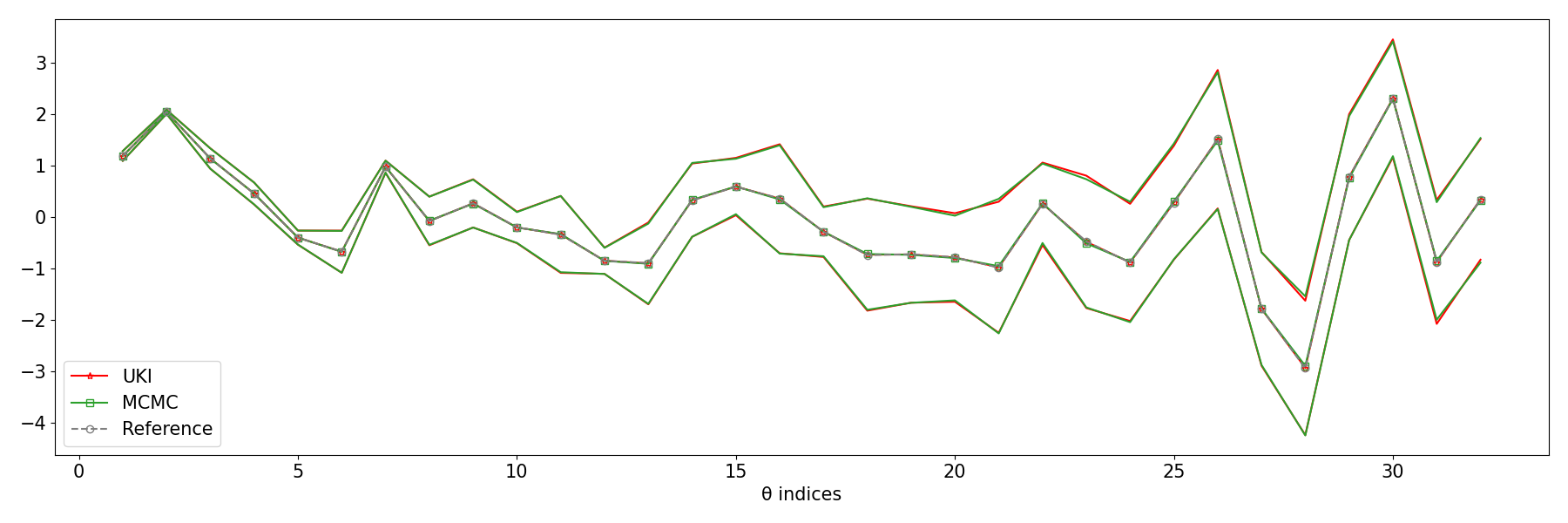}
    \caption{The estimated KL expansion parameters $\theta_{(i)}$ and the associated 3-$\sigma$ confidence intervals obtained by UKI and MCMC for the Darcy flow problem.}
    \label{fig:darcy-1d-uq1}
\end{figure}

\begin{figure}[ht]
\centering
    \includegraphics[width=0.98\textwidth]{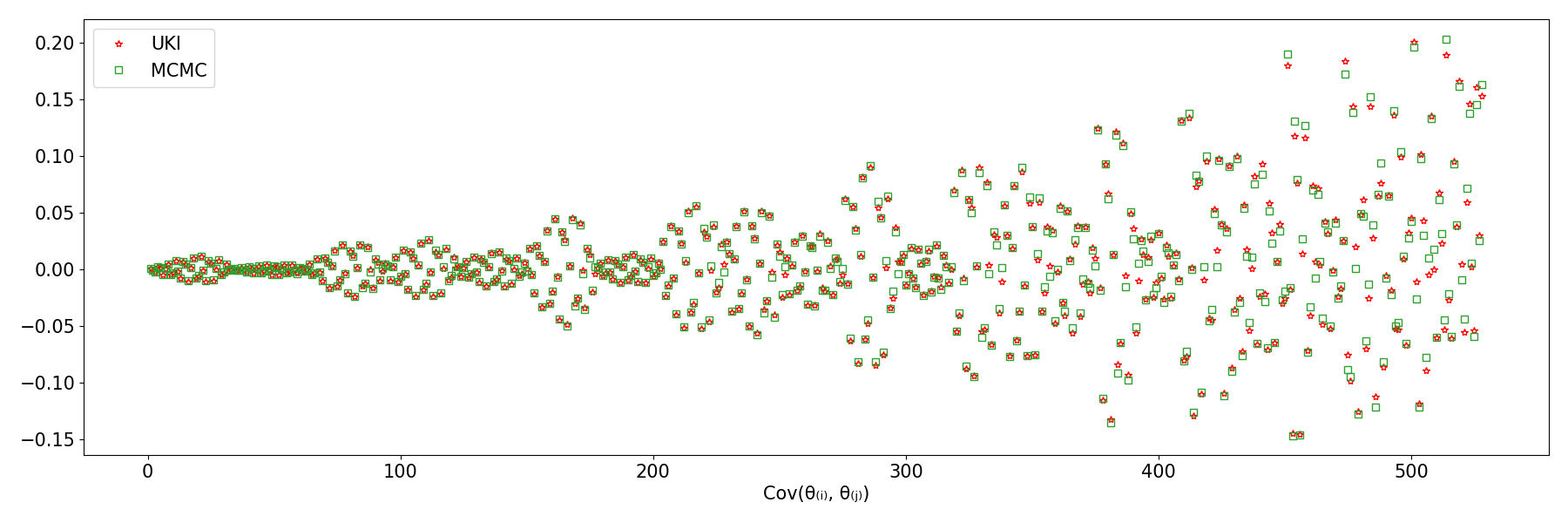}
    \caption{The estimated covariance for each parameter pair ($\theta_{(i)},  \theta_{(j)}$) obtained by UKI and MCMC for the Darcy flow problem.}
    \label{fig:darcy-1d-uq2}
\end{figure}

\subsection{Navier-Stokes Problem}
\label{ssec:app:NS}
Finally, we consider the 2D Navier-Stokes equation on a periodic domain $D = [0,2\pi]\times[0,2\pi]$:
\begin{equation*}
\begin{split}
    &\frac{\partial v}{\partial t} + (v\cdot \nabla) v + \nabla p - \nu\Delta v = 0, \\
    &\nabla \cdot v = 0, \\
    &\frac{1}{4\pi^2}\int v  = v_b,
    \end{split}
\end{equation*}
here $v$ and $p$ denote the velocity vector and the pressure, $\nu=0.01$ denotes the dynamic viscosity, and $v_b = (2\pi, 2\pi)$ denotes the non-zero mean background velocity.
The forward problem is rewritten in the vorticity-streamfunction~($\omega-\psi$) formulation:
\begin{equation*}
\begin{split}
    &\frac{\partial \omega}{\partial t} + (v\cdot\nabla)\omega - \nu\Delta\omega = 0, \\
    &\omega = -\Delta\psi \qquad \frac{1}{4\pi^2}\int\psi = 0,\\
    &v = \Big(\frac{\partial \psi}{\partial x_2}, -\frac{\partial \psi}{\partial x_1}\Big) + v_b,
\end{split}
\end{equation*}
and solved by the pseudo-spectral method~\cite{hesthaven2007spectral} on a $128\times128$ grid. To eliminating the aliasing error, the Orszag 2/3-Rule~\cite{orszag1972numerical} is applied and, therefore there are $85^2$ Fourier modes (padding with zeros). The time-integrator is the Crank–Nicolson method with $\Delta T=2.5\times 10^{-4}$. 

The random initial vorticity field $\omega_0(x, \theta)$, depending on parameters $\theta \in \R^{N_{\theta}}$, is modeled as a Gaussian random field with covariance operator $\mathsf{C} = \Delta^{-2}$, which subjects to periodic boundary conditions on the space of spatial-mean zero functions. The KL expansion of the initial vorticity field is given by 
\begin{equation}
\label{eq:NS-KL-2d}
\omega_0(x, \theta) = \sum_{l\in K} \theta^{c}_{(l)} \sqrt{\lambda_{l}} \psi^c_l  +  \theta^{s}_{(l)}\sqrt{\lambda_{l}} \psi^s_l,
\end{equation}
where $K = \{(k_x, k_y)| k_x + k_y > 0 \textrm{ or } (k_x + k_y = 0 \textrm{ and } k_x > 0)\}$, and the eigenpairs are of the form
\begin{equation*}
    \psi^c_l(x) =\frac{\cos(l\cdot x)}{\sqrt{2}\pi}\quad \psi^s_l(x) =\frac{\sin(l\cdot x)}{\sqrt{2}\pi} \quad \lambda_l = \frac{1}{|l|^{4}},
\end{equation*}
and $\theta^{c}_{(l)},\theta^{s}_{(l)}  \sim \N(0,2\pi^2)$ i.i.d. The KL expansion~\cref{eq:NS-KL-2d} can be rewritten as a sum over $\Z^{0+}$ rather than a lattice: 
\begin{equation}
\label{eq:NS-KL-1d}
    \omega_0(x,\theta) = \sum_{k\in \Z^{0+}} \theta_{(k)}\sqrt{\lambda_k} \psi_k(x),
\end{equation}
where the eigenvalues $\lambda_k$ are in descending order.

For the inverse problem, we recover the initial condition, specifically the initial vorticity field of the Navier-Stokes equation, given pointwise observations $y_{ref}$ of the vorticity field at 64 equidistant points~($N_y=128$) at $T=0.25$ and $T=0.5$~(See \cref{fig:NS-obs}). 
And $5\%$ Gaussian random noises are added to the observation, as follows,
\begin{equation*}
    y_{obs} = y_{ref} + 5\% y_{ref} \odot \N(0, \I),
\end{equation*}
here $\odot$ denotes element-wise multiplication. The observation error is $\eta \sim \N(0, 0.1^2\I)$, which corresponds to the noise level.
The initial vorticity field is generated with $100$ Fourier modes of~\cref{eq:NS-KL-2d}, and the first  50 modes  are recovered~($N_{\theta}=100$), and hence the model is misspecified and model error exists. The UKI is initialized with $\theta_0 \sim \N(0, \I)$.

\begin{figure}[ht]
\centering
\includegraphics[width=0.45\textwidth]{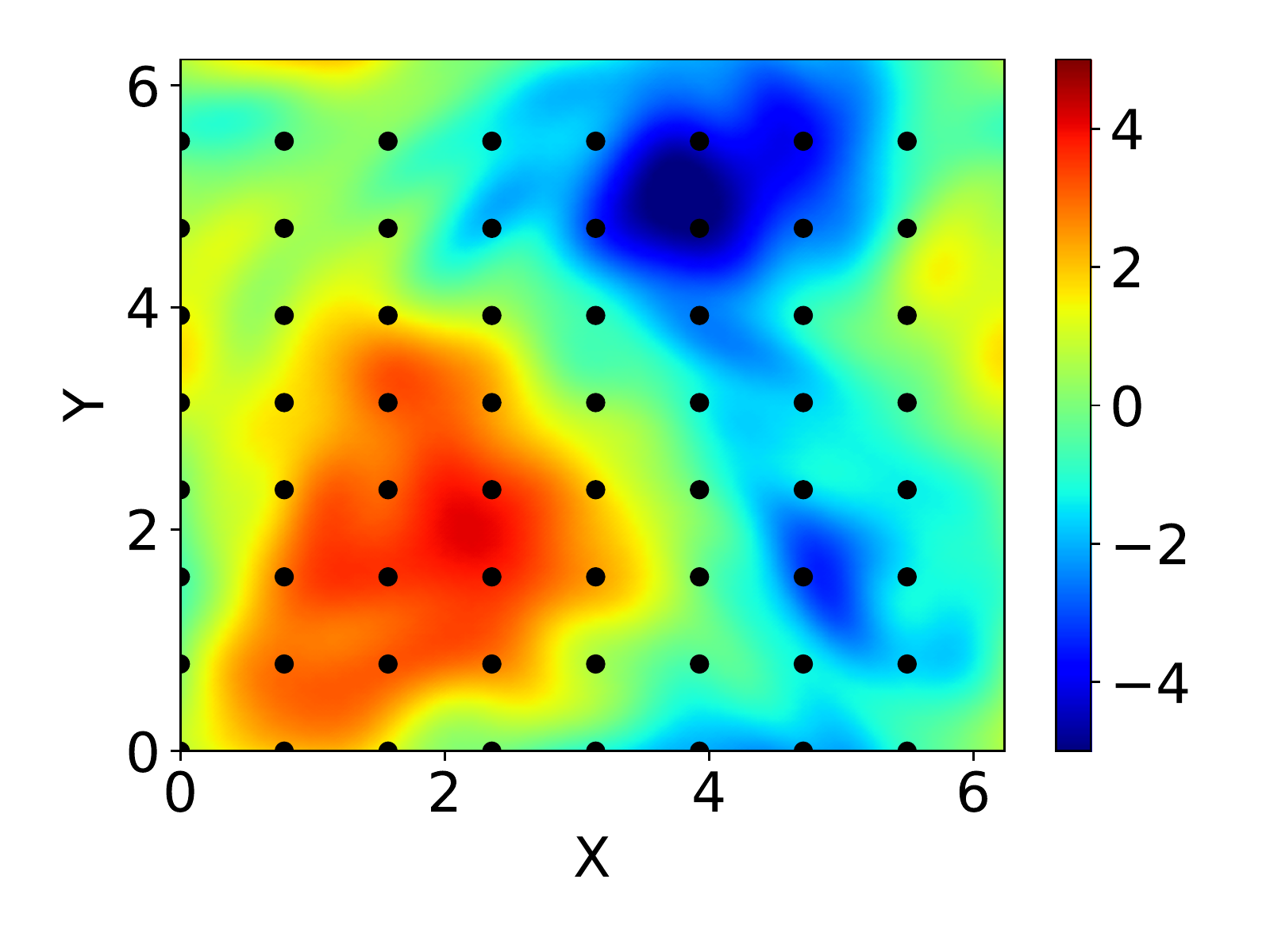}
\includegraphics[width=0.45\textwidth]{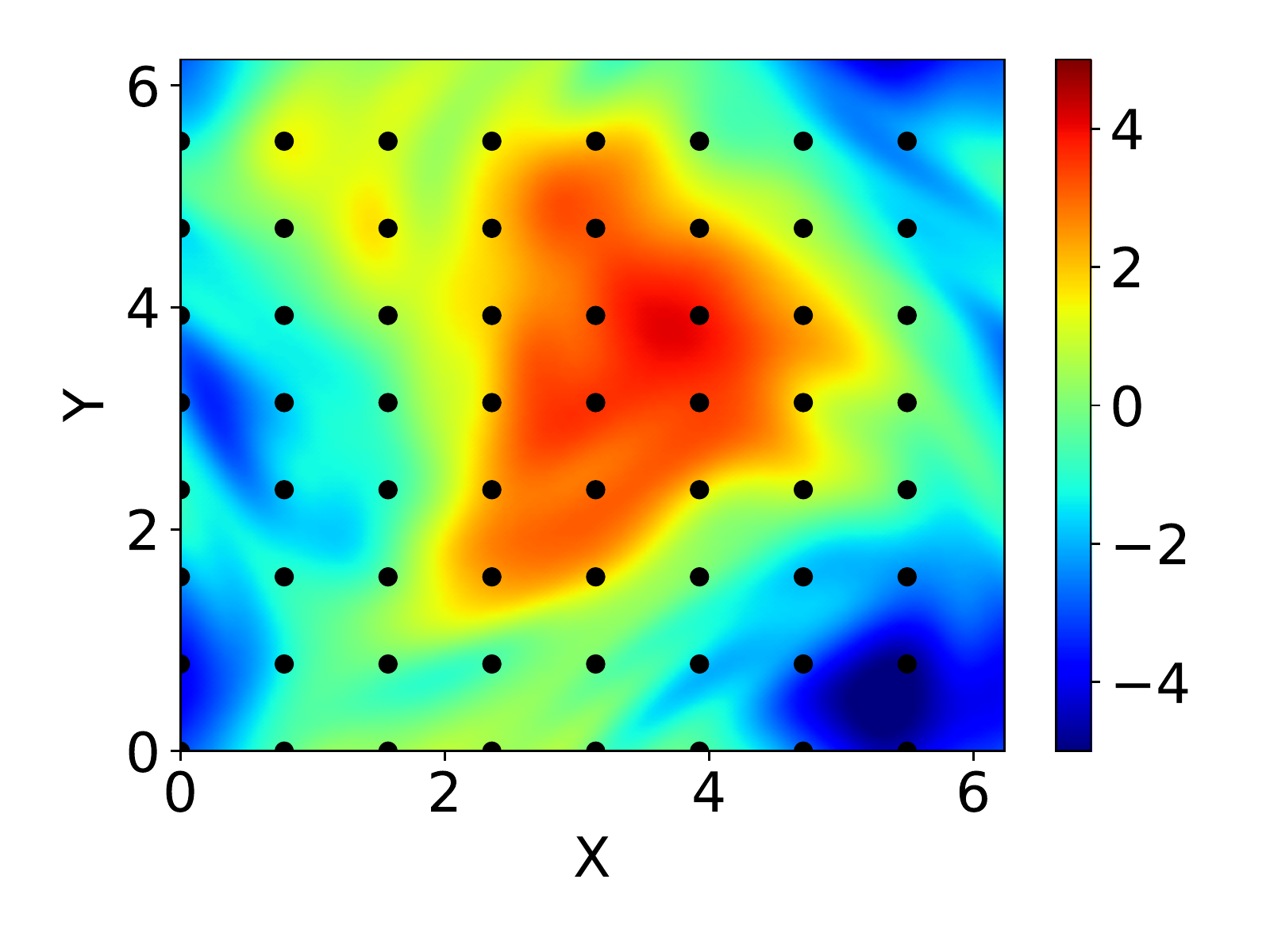}
\caption{The vorticity field of the Navier-Stokes problem and the $64$ equidistant pointwise measurements~(black dots) at two observation times~($T=0.25$ and $T=0.5$).}
\label{fig:NS-obs}
\end{figure}

The convergence of the initial vorticity field $\omega_0(x, \mean_n)$, the optimization errors, and the Frobenius  norm of the covariance are depicted in~\cref{fig:NS-UQ-Loss}. 
Thanks to the linear~(or superlinear) convergence rate of the LMA~\cite{UKI}, 
the UKI converges efficiently.

The truth random initial vorticity field $\omega_0(x,\theta_{ref})$ and the initial vorticity field recovered by UKI at the 20th iteration are depicted in~\cref{fig:NS-omega}. The UKI captures well main features of the truth random initial field and even small features, despite the irreversibility of the diffusion process~($\nu=0.01$).
The estimated parameters $\{\theta^{c}_{(l)}\}, \{\theta^{s}_{(l)}\}$, and the associated 3-$\sigma$ confidence intervals obtained by the UKI are depicted in~\cref{fig:NS-theta}. Most of the reference values are in the confidence interval.

\begin{figure}[ht]
\centering
    \includegraphics[width=0.98\textwidth]{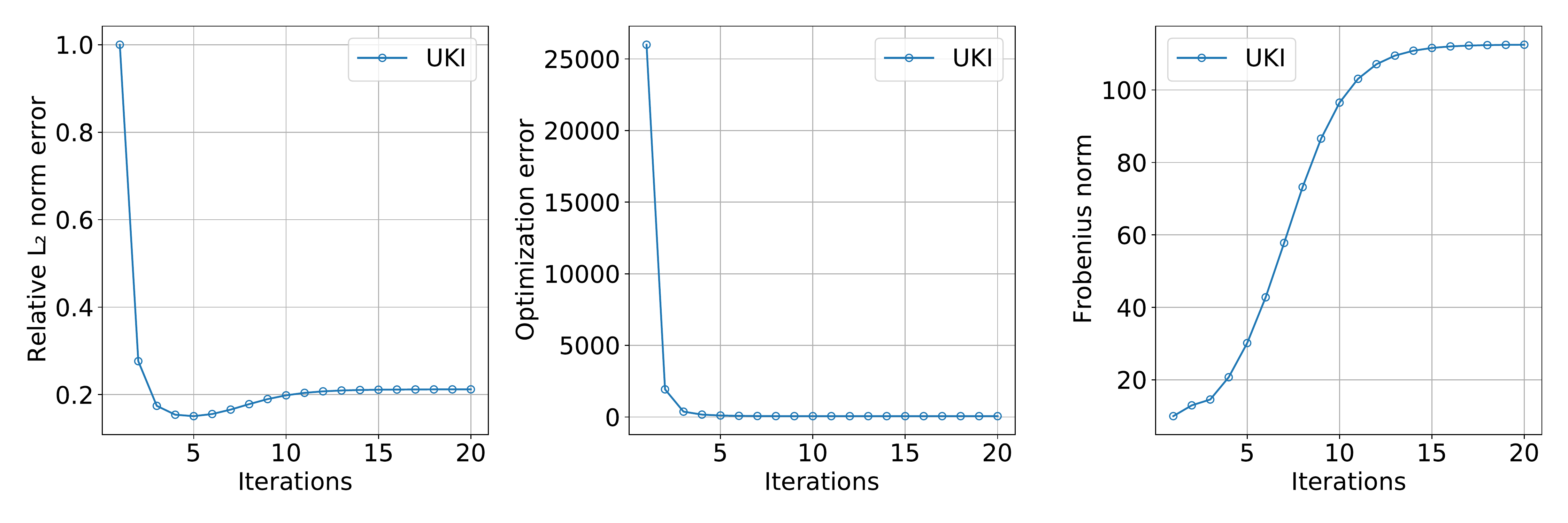}
    \caption{
    Relative error $\frac{\lVert\omega_0(x,\mean_n) - \omega_{0,ref}\rVert_2}{\lVert \omega_{0,ref}\rVert_2}$ (left), the optimization error $\displaystyle \frac{1}{2}\lVert\Sigma_{\nu}^{-\frac{1}{2}} (y_{obs} - \py_n)\rVert^2$~(middle), and the Frobenius norm of covariance $\lVert \Cov_n\rVert_F$ for the Naviar-Stokes initial inverse problem.}
    \label{fig:NS-UQ-Loss}
\end{figure}

\begin{figure}[ht]
\centering
    \includegraphics[width=0.45\textwidth]{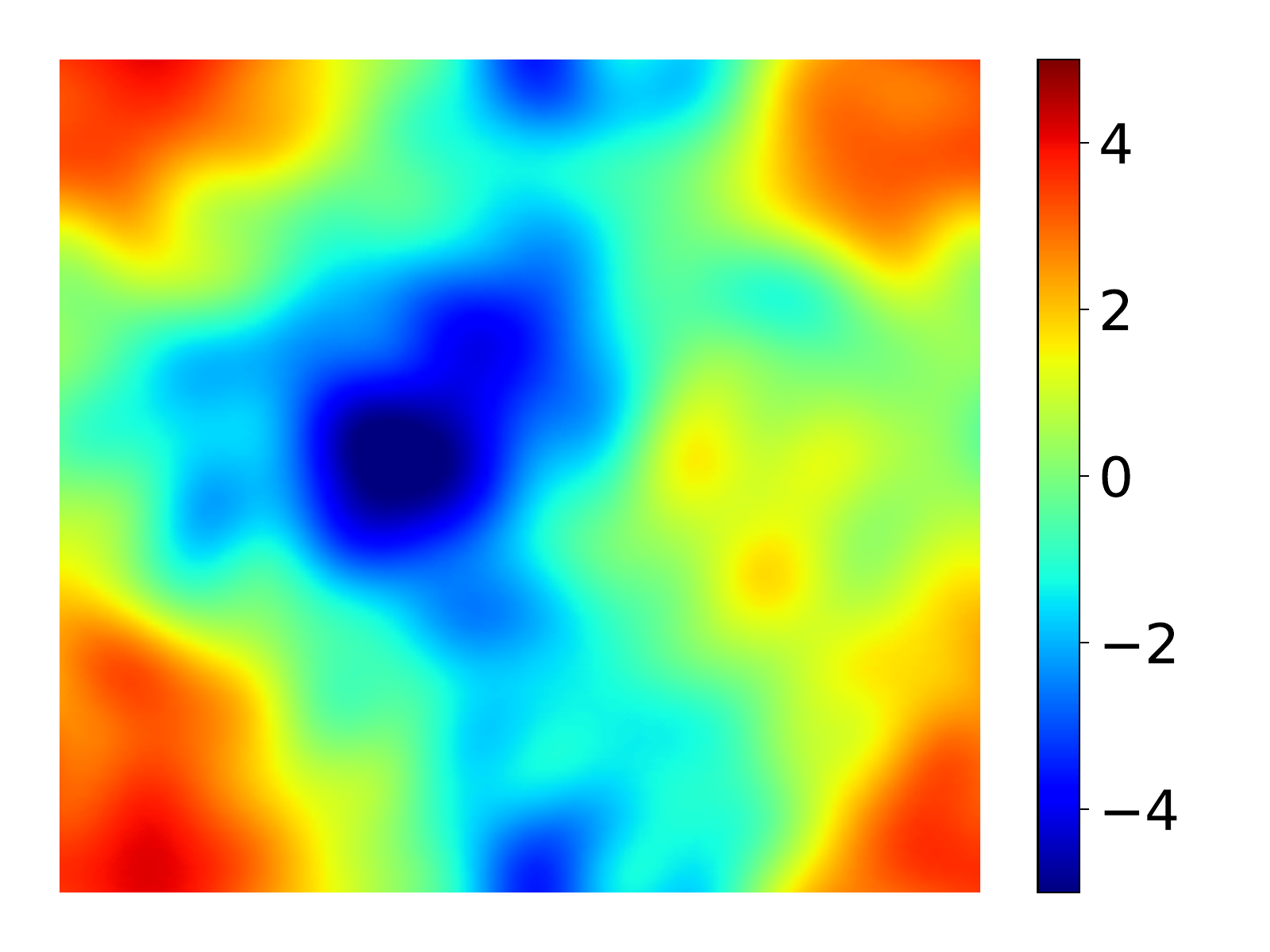}
    \includegraphics[width=0.45\textwidth]{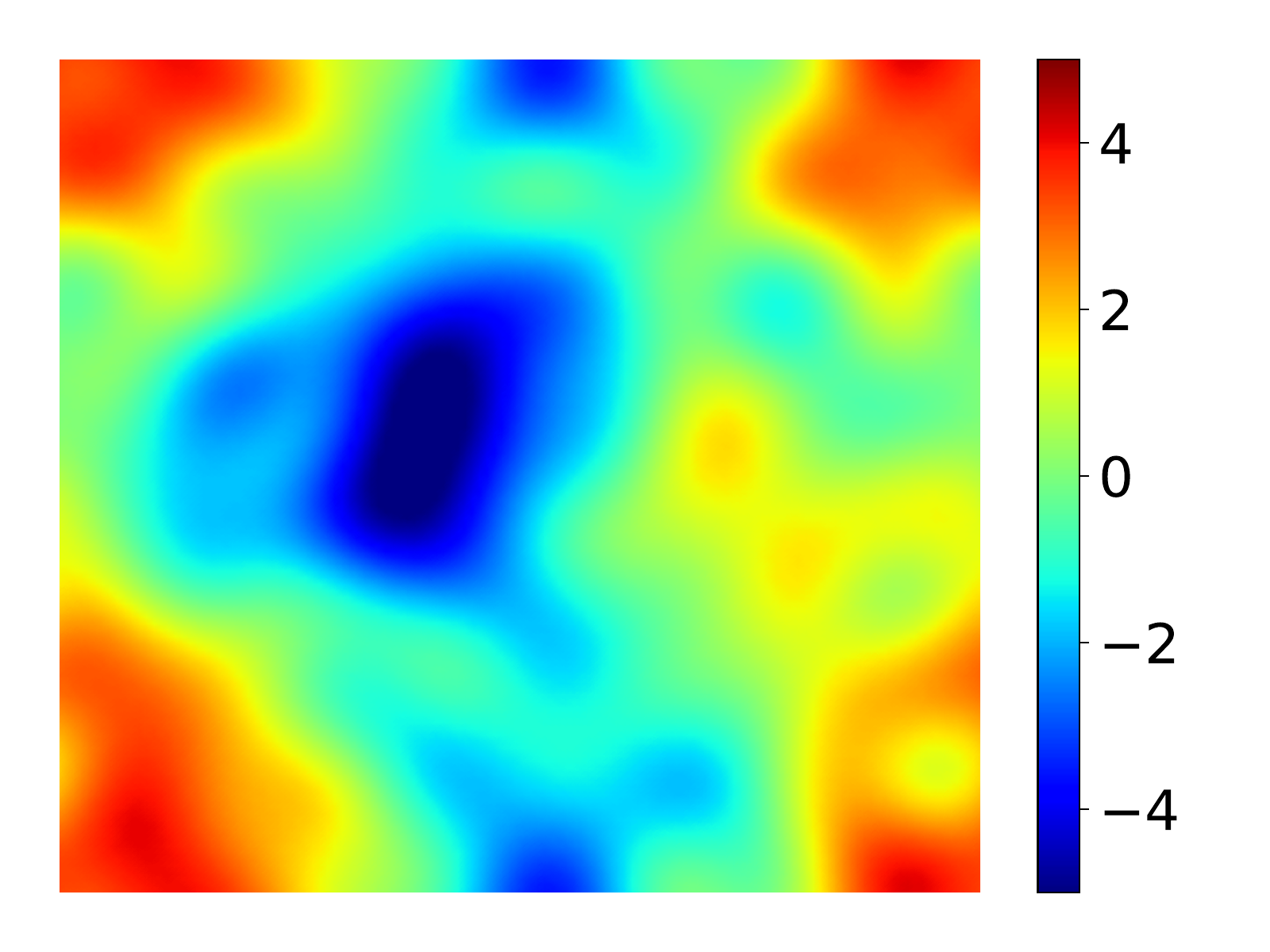}
    \caption{
    Truth initial vorticity field~$\omega_0(x, \theta_{ref})$~(left) and the Initial vorticity fields $\omega_0(x, \mean_n)$ recovered by the UKI~(right) for the Naviar-Stokes problem.}
    \label{fig:NS-omega}
\end{figure}

\begin{figure}[ht]
\centering
    \includegraphics[width=0.98\textwidth]{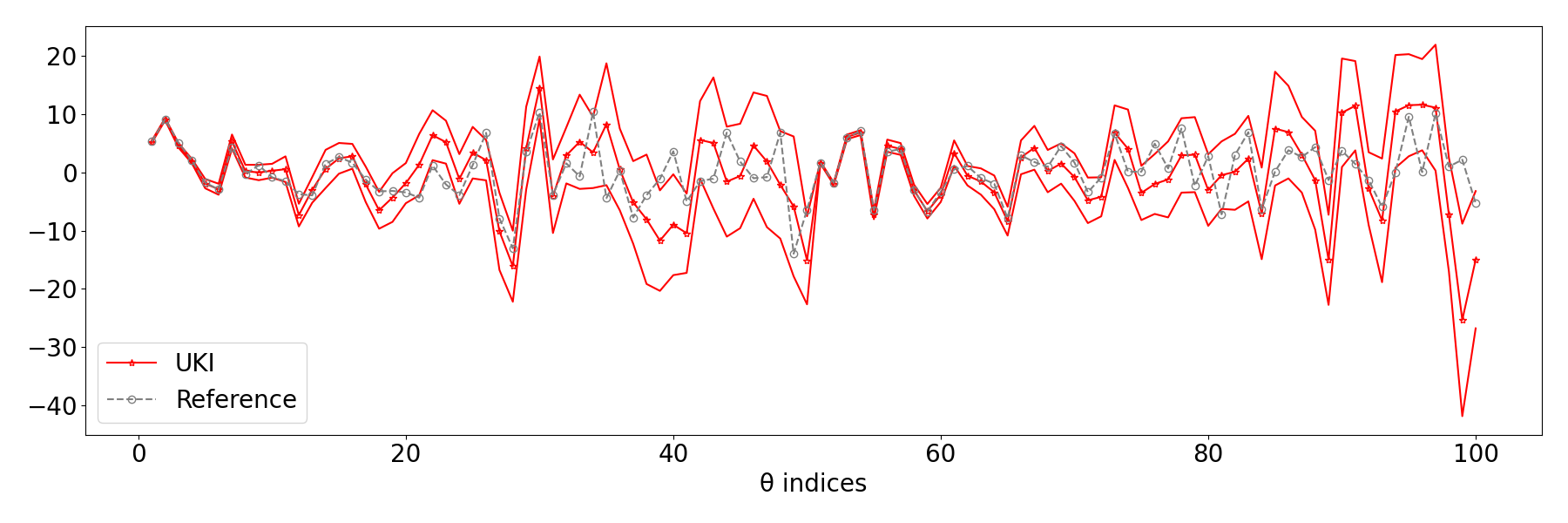}
    \caption{
    The estimated Fourier expansion parameters $\{\theta^{c}_{(l)}\}, \{\theta^{s}_{(l)}\}$ and the associated 3-$\sigma$ confidence intervals obtained by UKI for the Naviar-Stokes problem.}
    \label{fig:NS-theta}
\end{figure}

\section{Conclusion}
Unscented Kalman inversion is proposed as a derivative-free optimization method for inverse problems.
In this paper, we further study the capability of UKI  for Bayesian inference and uncertainty quantification.
We focus on well-posed inverse problems, since a wide range of inverse problems belong to this category, especially with the availability of large and diverse data sets from experiments and direct simulations.
Theoretical guarantees are presented, for linear inverse problems, the mean and covariance obtained by the UKI converge exponentially fast to the posterior mean and covariance;
for nonlinear invertible inverse problems, the error bounds of the mean and covariance are derived, in terms of the observation error.
Theoretical study of the UKI for nonlinear over-determined inverse problems is worth further investigation.
Numerical studies empirically confirm the theoretical results and demonstrate the effectiveness of UKI as an efficient Bayesian inference tool.
For the Navier-Stokes problem, the model error exists. It is interesting to systematically quantify model-form error by using UKI in the future.

\paragraph{Acknowledgments} D.Z.H. is supported by the generosity of Eric and Wendy Schmidt by recommendation of the Schmidt Futures program. J.H. is supported by the Simons Foundation as a Junior Fellow at New York University. The authors thank Prof. Sebastian Reich for helpful comments on the draft.

\appendix 
\section{Gaussian Approximation}
\label{sec:app:GA}
For the Gaussian approximation in~\cref{sec:KF_Intro}, assume the distribution of $\theta_n|Y_n$ is $\mu_{n} = \N(\mean_n, \Cov_n)$, then the distribution of $\theta_{n+1}|Y_n$ is $\hat{\mu}_{n+1} = \N(\pmean_{n+1}, \pCov_{n+1})$.
The joint density function of $\{\theta_{n+1}, y_{n+1}\}|Y_n$ is $p(\theta, y)$:
\begin{equation}
\label{eq:true-joint}
  \pf(\theta_{n+1}, y_{n+1}|Y_n) = \pf(y_{n+1}|\theta_{n+1}, Y_n) \pf(\theta_{n+1}|Y_n) \propto  e^{- \Phi (\theta_{n+1}, y_{n+1})} \pf_0(\theta_{n+1}),
\end{equation}
where $p_0$ is the density function of $\theta_{n+1}|Y_n$.
The Gaussian approximated joint density function is $p^G(\theta, y)$
\begin{equation}
    \pf^G(\theta_{n+1}, y_{n+1}| Y_n) =
    \N\Big(
    \begin{bmatrix}
    \pmean_{n+1}\\
    \py_{n+1}
    \end{bmatrix}, 
    \begin{bmatrix}
  \pCov_{n+1} & \pCov_{n+1}^{\theta p}\\
    {{\pCov_{n+1}}^{\theta p}}{}^{T} & \pCov_{n+1}^{pp}
    \end{bmatrix}
    \Big),
\end{equation}
It is worth mentioning $p(\theta, y)$ and $p^G(\theta, y)$ have the same mean and covariance by the definition of $p^G$.
To obtain an upper bound for the Kullback–Leibler divergence between the true distribution $\theta_{n+1}|Y_{n+1}$ and $\N(\mean_{n+1}, \Cov_{n+1})$, we need the following theorem:

\begin{theorem}
\label{thereom:KL_bound}
Let denote the joint density function of $x = (\theta\,;\,y)$
% hat means joint distribution 
% no-hat means conditional distribution
\begin{equation}
\label{eq:the-p}
    \pf(x) = \frac{1}{\hat{Z}} e^{-Q(x) - \epsilon H(x)} \quad \textrm { with } \quad 
    Q(x) = \frac{1}{2} (
    x - \pmean)^T \pCov^{-1} (x - \pmean),
\end{equation}
here $Q(x)$ represents the quadrature term, $\epsilon H(x)$ represents other terms and $\hat{Z}$ is the normalization constant.
Let $\pmean^G$ and $\pCov^G$ denote the mean and covariance of $\pf(x)$, 
the approximated joint Gaussian density function is 
\begin{equation*}
    \pf^G(x) = \frac{1}{\hat{Z}^G} e^{-Q^G(x)}  \quad \textrm { with } \quad 
    Q^G (x) = \frac{1}{2} (
    x - \pmean^G)^T {\pCov}^{G^{-1}} (x - \pmean^G).
\end{equation*}
We assume  
\begin{enumerate}
    \item $\exists~c_0 > 0 \textrm{ and } c_1 > 0$, such that  $| H(x) |\leq c_0 x^T x + c_1$, and $\lambda_{min}(\pCov^{-1}) > 2 \epsilon c_0$,\footnote{$\lambda_{min}(A)$ denotes the smallest eigenvalue of matrix $A$.}
    \item$\exists~c_2 > 0$, such that $\lambda_{min} (\pCov) > c_2$, $\lambda_{min} (\pCov^{G^{-1}}) > c_2$, and $\lambda_{min} (\pCov^{G}) > c_2$,
    \item dim$(x)$, $c_i$, and each entry of $y$, $\pmean$, $\pCov$, $\pCov^{-1}$, $\pmean^G$, $\pCov^G$, $\pCov^{G^{-1}}$ are $\bigO(1)$ constants, and $\epsilon$ is small enough.
\end{enumerate}
We have an upper bound for the Kullback–Leibler divergence between the conditional distributions:
\begin{equation}
\label{eq:KL-bound}
    \begin{split}
        \mathcal{D}_{KL}\Big( \pf^{G}(\theta | y) \, \Big|\Big| \, \pf(\theta | y) \Big) = \bigO(\epsilon).
    \end{split}
\end{equation}
\end{theorem}

\begin{proof}
We first prove
\begin{equation}
\label{eq:lamma1}
    \lVert \pmean^G - \pmean \rVert_{\infty} = \bigO(\epsilon) \quad \textrm{ and } \quad \lVert \pCov^G - \pCov \rVert_{\infty} = \bigO(\epsilon).
\end{equation}
We denote a Gaussian distribution $\pf_0(x)$ with mean $\pmean$ and covariance $\pCov$, and its associated expectation $\hat{\E}[\cdot]$ as following,
\begin{equation*} 
\pf_0(x) = \frac{1}{Z_0} e^{-Q(x)} \qquad  \hat{\E}[f] = \int f(x) \pf_0(x)   dx.
\end{equation*}
Since we consider only mean and covariance, we restrict $f$ to be $x$ and $xx^T$ in the following discussion.
The expectation with respect to the density function $\pf(x)$ in \cref{eq:the-p} can be written as
\begin{equation}
    \int  f(x) \pf(x) dx = \frac{\hat{\E} f e^{-\epsilon H(x)}}{\hat{\E} e^{-\epsilon H(x)}},
\end{equation}
which leads to  
\begin{equation}
\label{eq:mean_cov_diff}
\begin{split}
      \int f(x) \pf(x) - f(x) \pf_0(x) d\theta 
    = &\frac{\hat{\E}[ f(x) e^{-\epsilon H(x)} ]}{\hat{\E} [ e^{-\epsilon H(x)} ]} - \hat{\E} [ f(x) ] \\
    = &\int_{0}^{\epsilon} \partial_t \frac{\hat{\E}[ f(x) e^{-t H(x)} ]}{\hat{\E} [ e^{-t H(x)} ]} dt \\
     = &\int_{0}^{\epsilon} \frac{\hat{\E}[ - f(x) H(x) e^{-t H(x)} ]}{\hat{\E} [ e^{-t H(x)} ]} +  
     \frac{\hat{\E}[ f(x) e^{-t H(x)} ] \hat{\E} [ H(x) e^{-t H(x)} ] }{\hat{\E} [ e^{-t H(x)} ]^2} dt.
\end{split}
\end{equation}
Since $\lambda_{min}(\pCov^{-1}) > 2\epsilon c_0$, $\pf_0(x) e^{\epsilon c_0x^Tx}$ decays exponentially. For $t\in[0,\,\epsilon]$, we have
\begin{equation}
\label{eq:exp_bound}
\begin{split}
    &\Big|\frac{1}{ \hat{\E}[e^{-t H(x)} ] }\Big|  \leq \frac{1}{\hat{\E}[ e^{-\epsilon c_0x^Tx - \epsilon c_1} ]} = \bigO(1),  \\
    &\Big\lVert  \hat{\E}[f(x)e^{-t H(x)} ] \Big\rVert_{\infty}  \leq  \hat{\E}[ (\lVert x\rVert_{\infty} + N_{\theta}x^Tx)e^{\epsilon c_0x^Tx  + \epsilon c_1}  ] = \bigO(1), \\
    &\Big\lVert   \hat{\E}[H(x)e^{-t H(x)} ] \Big\rVert_{\infty}  \leq  \hat{\E}[ ( c_0 x^Tx + c_1 )e^{\epsilon c_0x^Tx  + \epsilon c_1}  ] = \bigO(1), \\
    &\Big\lVert  \hat{\E}[f(x)H(x)e^{t H(x)} ] \Big\rVert_{\infty}  \leq  \hat{\E}[  (\lVert x\rVert_{\infty} + N_{\theta}x^Tx)( c_0 x^Tx + c_1 ) e^{\epsilon c_0x^Tx + \epsilon c_1}  ] = \bigO(1).
\end{split}
\end{equation}
Plugging \cref{eq:exp_bound} into \cref{eq:mean_cov_diff} leads to 
\begin{equation}
\begin{split}
      \Big\lVert \int f(x) \pf(x) - f(x) \pf_0(x) dx \Big\rVert_{\infty}
    = \bigO(\epsilon),
\end{split}
\end{equation}
which finishes the proof of \cref{eq:lamma1}.
%%%%%%%%%%%%%%%%%%%%%%%%%%%%%%%%%%%%%%%%%%%%%%%%%%%%%%%%%%%%%%%%%%%%%%%%%%%%%%%%%%%%%%%%%%

Next, we define conditional distributions
\begin{equation*}
    \pf(\theta|y) = \frac{\pf(x)}{\int \pf(x) dy} := \frac{1}{Z(y)}  e^{-Q(\theta|y) - \epsilon H(\theta, y)} 
    \quad \textrm{and} \quad
    \pf^G(\theta|y) = \frac{\pf^G(x)}{\int \pf^G(x) dy} := \frac{1}{Z^G(y)} e^{-Q^G(\theta| y)}, 
\end{equation*}
where 
\begin{equation*}
 Q(\theta | y) = \frac{1}{2} (
    \theta - \mean)^T \Cov^{-1} (\theta - \mean)
    \quad \textrm{and}\quad
    Q^G (\theta | y) = \frac{1}{2} (
    \theta - \mean^G)^T {\Cov}^{G^{-1}} (\theta - \mean^G).
\end{equation*}
The conditional normal distribution theorem leads to 
\begin{equation}
  \begin{split}
  \label{eq:mean_cov}
&\mean = \pmean_1 + \pCov_{12}\pCov_{22}^{-1}(y - \pmean_2)   \qquad 
\textrm{and} \qquad \Cov = \pCov_{11} - \pCov_{12}\pCov_{22}^{-1}\pCov_{21}, \\
&\mean^G = \pmean^G_1 + \pCov^G_{12}{\pCov^{G^{-1}}_{22}}(y - \pmean^G_2)    
\qquad \textrm{and} \qquad \Cov^G = \pCov^G_{11} - \pCov^G_{12}{\pCov^{G^{-1}}_{22}}\pCov^G_{21},
\end{split}  
\end{equation}
here the subscripts $1$ and $2$ correspond to $\theta$ and $y$ components, respectively.

We then prove
\begin{equation}
\label{eq:lamma2}
    \lVert \mean^G - \mean \rVert_{\infty} = \bigO(\epsilon) \quad \textrm{ and } \quad \lVert \Cov^G - \Cov \rVert_{\infty} = \bigO(\epsilon).
\end{equation}
Since $\lambda_{min}(\pCov ) > c_2$ and $\lambda_{min}(\pCov^{G} ) > c_2$, we have
\begin{equation}
\label{eq:lambda_min1}
    \lambda_{min}(\pCov_{22} ) = \min_{y^Ty=1} y^T \pCov_{22}y = 
    \min_{y^Ty=1, \theta = 0} (\theta\,;\,y)^T \pCov (\theta\,;\,y) 
    \geq \min_{x^Tx=1} x^T \pCov x  =  \lambda_{min}(\pCov ) > c_2,
\end{equation}
and 
\begin{equation}
\label{eq:lambda_min2}
\begin{aligned}
    \lambda_{min}(\pCov^G_{22} ) = \min_{y^Ty=1} y^T \pCov_{22}^{G}y = 
    \min_{y^Ty=1, \theta = 0} (\theta\,;\,y)^T \pCov^{G} (\theta\,;\,y) 
    \geq \min_{x^Tx=1} x^T \pCov^{G} x  =  \lambda_{min}(\pCov^{G} ) > c_2.
\end{aligned}
\end{equation}
Bringing \cref{eq:lambda_min1,eq:lambda_min2,eq:lamma1} into \cref{eq:mean_cov} leads to \cref{eq:lamma2}.

Moreover, since the Schur complement $\Cov$ in~\cref{eq:mean_cov} satisfies  
%https://math.stackexchange.com/questions/2749464/is-schur-complement-better-conditioned-than-the-original-matrix
\begin{equation}
\label{eq:shur-0}
\begin{aligned}
  \lambda_{max}(\pCov) =   \lambda_{max}\Big(
  \begin{bmatrix}
     C & \\
      & 0
  \end{bmatrix} +  
  \begin{bmatrix}
     \pCov_{12}\pCov_{22}^{-{1}/{2}}\\
     \pCov_{22}^{{1}/{2}}
  \end{bmatrix} 
  \begin{bmatrix}
     \pCov_{22}^{-{1}/{2}} \pCov_{21} &
     \pCov_{22}^{{1}/{2}}
  \end{bmatrix} 
  \Big) \geq \lambda_{max}(\Cov)
\end{aligned},
\end{equation}
we have 
\begin{equation}
\label{eq:shur}
\begin{aligned}
      \lambda_{min}(\Cov^{-1}) \geq  \lambda_{min}(\pCov^{-1}) > 2\epsilon c_0 \qquad 
      \lambda_{min}(\Cov^{G^{-1}}) \geq  \lambda_{min}(\pCov^{G^{-1}}) > c_2.
\end{aligned}
\end{equation}

Finally, the Kullback–Leibler divergence can be written as 
\begin{equation}
\label{eq:KL}
  \int \pf^G(\theta | y) \log\Big(\frac{\pf^G(\theta|y)}{\pf(\theta|y)}\Big) d\theta = \int \pf^G(\theta | y) \Big( Q(\theta|y) + \epsilon H(\theta, y) - Q^G(\theta| y) + \log{Z} - \log{Z^G} \Big) d\theta. 
\end{equation}
The first part of \cref{eq:KL} can be written as 
\begin{equation}
\label{eq:KL_1}
\begin{split}
    & \int \pf^G(\theta | y) \Big( Q(\theta| y) + \epsilon H(\theta, y) - Q^G(\theta|y) \Big) d\theta\\
    =&  \int \pf^G(\theta | y)  \epsilon H(\theta, y) d\theta + \int \pf^G(\theta | y)\Big( Q(\theta|y) - Q^G(\theta|y) \Big) d\theta\\
    =&  \int \pf^G(\theta | y)  \epsilon H(\theta, y) d\theta + \frac{1}{2} tr(C^{-1}C^G) + \frac{1}{2}(\mean^G - \mean)^T C^{-1} (\mean^G - \mean) - \frac{1}{2} tr( C^{G^{-1}}C^G)\\
    =&  \int \pf^G(\theta | y)  \epsilon H(\theta, y) d\theta + \frac{1}{2} tr\Big( C^{-1} (C^G-C) \Big) + \frac{1}{2}(\mean^G - \mean)^T C^{-1} (\mean^G - \mean)\\
    =& \bigO(\epsilon).
\end{split}
\end{equation}
For the second part of \cref{eq:KL}, 
let denote $Z_0 = \int e^{-Q(\theta| y)} d\theta$ and the expectation $\E[\cdot]$ with respect to the Gaussian distribution 
$$ \pf_0(\theta|y) = \frac{1}{Z_0(y)} e^{-Q(\theta| y)}.$$
Combining \cref{eq:shur}  and \cref{eq:exp_bound} leads to 
\begin{equation}
\label{eq:Z_bound}
\begin{split}
    &\log{Z} - \log{{Z_0}} = \log \E[e^{-\epsilon H(\theta, y)}] 
    = \int_0^{\epsilon}  \partial_t \log \E[e^{-t H(\theta, y)}]  dt
    = \int_0^{\epsilon} \frac{\E[-H(\theta, y) e^{-tH(\theta, y)}]}{\E [e^{-tH(\theta, y)}] } dt = \bigO(\epsilon), \\
  &\log{Z_0} - \log{Z^G} = \frac{1}{2}\Big( \log \det{\Cov^G} - \log \det{\Cov} \Big) = \bigO(\epsilon).
\end{split}
\end{equation}
Bringing \cref{eq:Z_bound} into  the second part of \cref{eq:KL} leads to 
\begin{equation}
\label{eq:KL_2}
\begin{split}
    \int \pf^G(\theta | y) \Big( \log{Z} - \log{Z^G} \Big) d\theta = (\log{Z} - \log{{Z_0}}) + (\log{{Z_0}} - \log{Z^G}) = \bigO(\epsilon).\\
\end{split}
\end{equation}
Combining \cref{eq:KL_1} and \cref{eq:KL_2} leads to the upper bound of the Kullback–Leibler divergence.

\end{proof}

Following~\cref{thereom:KL_bound}, we assume $\Phi$ can be decomposed into the quadratic part~$Q^{\Phi}(\theta, y)$ and the high order part $\epsilon H(\theta, y)$, as following
\begin{equation*}
    \Phi(\theta, y) = Q^{\Phi}(\theta, y) + \epsilon H(\theta, y),
\end{equation*}
which is true when $\G(\theta)$ is close to a linear function. 
We further define 
$$\displaystyle Q(\theta, y) = Q^{\Phi}(\theta, y) + \frac{1}{2}(\theta - \pmean_{n+1})^T\pCov_{n+1}^{-1}(\theta - \pmean_{n+1}).$$
When $Q$ and $H$ satisfy the conditions in \cref{thereom:KL_bound}, the conditional distribution $\N(\mean_{n+1}, \Cov_{n+1})$ 
well approximates the true distribution $\theta_{n+1}|Y_{n+1}$.

\begin{remark}
When $\G$ is linear, namely $\Phi(\theta, y)$ is quadratic and $\epsilon=0$, $\N(\mean_{n+1}, \Cov_{n+1})$  is the exact distribution of $\theta_{n+1}|Y_{n+1}$. 
\end{remark}

\section{Proof of Theorem~\ref{th:lin_converge}}
\label{sec:app:Linear-UKI}
\begin{proof}

With the hyperparameters defined in~\cref{eq:hyperparameters}, the update equation of $\{\Cov_n\}$ in \cref{eq:Lin_KF_analysis} can be rewritten as 
\begin{equation}
\label{eq:Lin_KF_Cinv}
\begin{split}
    &\Cov_{n+1}^{-1} = G^T\Sigma_{\nu}^{-1}G + (\Cov_n + \Sigma_{\omega})^{-1} = \frac{1}{2}G^T\Sigma_{\eta}^{-1}G + (2\Cov_n)^{-1}.\\
\end{split}
\end{equation}
We have a close formula for $\Cov^{-1}_{n}$: 
\begin{equation}
\label{eq:beta_1}
    \Cov_n^{-1} = 
    \Big[1 - \frac{1}{2^n}\Big] G^T\Sigma_{\eta}^{-1}G + \frac{1}{2^n} \Cov_0^{-1}.
\end{equation}
This leads to the exponential convergence $\displaystyle \lim_{n\to \infty} \Cov_n^{-1} = G^T\Sigma_{\eta}^{-1}G$.

The convergence proof of $\mean_n$ basically follows~\cite{UKI}. \Cref{eq:Lin_KF_Cinv,eq:beta_1} lead to 
\begin{equation}
\label{eq:C_bound}
\begin{split}
\frac{1}{2}G^T\Sigma_{\eta}^{-1}G = G^T\Sigma_{\nu}^{-1}G \preceq\Cov_{n+1}^{-1} \preceq G^T\Sigma_{\nu}^{-1}G + \Sigma_{+}
    \quad \textrm{where} \quad 
    \Sigma_{+} = 
     G^T\Sigma_{\nu}^{-1}G + \Cov_0^{-1}. 
\end{split}
\end{equation}
The update equation of $\mean_n$ in \cref{eq:Lin_KF_analysis} can be rewritten as 
\begin{equation}
\mean_{n+1} = \mean_{n} + \Cov_{n+1}G^T\Sigma_{\nu}^{-1}(y - G\mean_{n}).
\end{equation}
Consider the Range($G^T$) $\otimes$ Ker($G$) decomposition of $\mean_n = \mean_n^{\parallel} + \mean_n^{\bot}$ with projections $P_{G^{\parallel}}$ and $P_{G^{\bot}}$, we have 
\begin{subequations}
\begin{align}
    &\mean_{n+1}^{\parallel} = \mean_{n}^{\parallel} + P_{G^{\parallel}}\Cov_{n+1}G^T\Sigma_{\nu}^{-1}(y - G\mean_{n}^{\parallel}), \label{eq:contracting-1} \\
    &\mean_{n+1}^{\bot} = \mean_n^{\bot} +  P_{G^{\bot}} \Cov_{n+1}G^T\Sigma_{\nu}^{-1}(y - G\mean_{n}^{\parallel}). \label{eq:contracting-2}
\end{align}
\end{subequations}
Constraining on Range($G^T$), we have the fact that $B := G^T\Sigma_{\nu}^{-1}G$ is symmetric and $B \succ 0$. From this, it follows that $\I -\pCov_{n+1} B$ has the same spectrum as $\I - B^{\frac{1}{2}}\Cov_{n+1}B^{\frac{1}{2}}$. Using the bounds on $\Cov_{n+1}$ appearing in~\cref{eq:C_bound}, the spectral radius of the update matrix in \cref{eq:contracting-1} satisfies
\begin{equation}
\label{eq:cov-mean-par}
\begin{split}
\rho(P_{G^{\parallel}} - P_{G^{\parallel}}\Cov_{n+1}G^T\Sigma_{\nu}^{-1}G) &\leq \rho(P_{G^{\parallel}}) \rho\Big(\I - \sqrt{G^T\Sigma_{\nu}^{-1}G}\Cov_{n+1}\sqrt{G^T\Sigma_{\nu}^{-1}G}\Big) \\
&\leq
1 - \rho\Big(\sqrt{G^T\Sigma_{\nu}^{-1}G}\big(G^T\Sigma_{\nu}^{-1}G + \Sigma_{+}\big)^{-1}\sqrt{G^T\Sigma_{\nu}^{-1}G}\Big) \\
&= 1 -\epsilon_0,
\end{split}
\end{equation}
where $\epsilon_0 \in (0,1)$. Hence, we have $\{\mean_{n}^{\parallel}\}$ converges exponentially to $\mean_{\infty}^{\parallel}$, which satisfies $\displaystyle G^T\Sigma_{\nu}^{-1}(y - G\mean_{\infty}^{\parallel}) = 0$. The update equation~\cref{eq:contracting-2} of $m_n^{\bot}$ can be written as
\begin{equation}
\begin{split}
    &\mean_{n+1}^{\bot} = \mean_n^{\bot} +  P_{G^{\bot}} \Cov_{n+1}G^T\Sigma_{\nu}^{-1}G(\mean_{\infty}^{\parallel} - \mean_{n}^{\parallel}).
\end{split}
\end{equation}
Since $\displaystyle \rho(P_{G^{\bot}} \Cov_{n+1}G^T\Sigma_{\nu}^{-1}G) \leq  \rho(P_{G^{\bot}}) \leq 1$ and $\displaystyle \lim_{n\rightarrow \infty} m_n^{\parallel} - \mean_{\infty}^{\parallel} = 0$ converges exponentially fast, we have the exponential convergence of  $\{\mean_n^{\bot}\}$ to $\mean_{\infty}^{\bot}$. Therefore, the converged vector $\mean_{\infty} = \mean_{\infty}^{\parallel} + \mean_{\infty}^{\bot}$ satisfies  $\displaystyle G^T\Sigma_{\nu}^{-1}(y - G\mean_{\infty}) = 0$, which is a minimizer of $\Phi$.
\end{proof}

\section{Proof of Theorem \ref{theorm:nonlinear}}
\label{sec:app:theorm:nonlinear:proof}
\begin{proof}
Since we consider only mean and covariance, we consider function $f(\theta) = \theta \textrm{ or } \theta \theta^T$. The expectation with respect to the posterior distribution $p(\theta|y)$ is  
\begin{equation}
\label{eq:exp_posterior}
\begin{split}
  \E[f(\theta)|y]  &= \int \frac{1}{Z} e^{- \Phi (\theta, y)} f(\theta) d\theta \\
                    &= \frac{1}{Z} \int  e^{- \frac{1}{2}(y - \G(\theta))^T \Sigma^{-1}_{\eta} (y - \G(\theta))} f(\theta) d\theta \\ 
                    &=  \frac{1}{Z} \int  e^{- \frac{1}{2}(y - \theta^{'})^T \Sigma^{-1}_{\eta} (y - \theta^{'})} f(\G^{-1}(\theta^{'})) \Big| \texttt{det}\frac{d\G^{-1}(\theta^{'})}{d\theta}\Big| d\theta' \textrm{ where } \theta' = \G(\theta).
\end{split}
\end{equation}
The expectation with respect to the pull-back distribution is 
\begin{equation}
\label{eq:exp_inverse_transform}
\begin{split}
    \E[f(\theta^{\eta})]  &= \frac{1}{Z_{\eta}}\int f(\G^{-1}(y - \eta)) e^{-\frac{1}{2}\eta^T \Sigma_{\eta}^{-1} \eta } d\eta \\
                     &= \frac{1}{Z_{\eta}}\int f(\G^{-1}(\theta^{'})) e^{-\frac{1}{2}(y - \theta^{'})^T \Sigma_{\eta}^{-1} (y - \theta^{'}) } d\theta^{'} \textrm{ where } \theta' = y - \eta.\\
\end{split}
\end{equation}
here $Z_{\eta} = \sqrt{(2\pi)^{N_y} \texttt{det}\Sigma_{\eta}}$.

Let denote $\E_{\theta}[\cdot]$ the expectation with respect to the Gaussian density function $\N(y, \Sigma_{\eta})$: 
\begin{equation}
\begin{split}
  &\E_{\theta}[f] = \frac{1}{Z_{\eta}} \int e^{- \frac{1}{2}(y - \theta)^T \Sigma^{-1}_{\eta} (y - \theta)}  f(\theta)   d\theta, \\
\end{split}
\end{equation}
The difference between these two expectations~\cref{eq:exp_posterior} and~\cref{eq:exp_inverse_transform} becomes

\begin{equation}
\label{eq:exp_diff}
\begin{split}
\E[f(\theta)|y] -   \E[f(\theta^{\eta})] =& \frac{1}{Z_{\eta}}\int  e^{- \frac{1}{2}(y - \theta^{'})^T \Sigma^{-1}_{\eta} (y - \theta^{'})} f(\G^{-1}(\theta^{'})) \Big(\frac{\Big| \texttt{det}\frac{d\G^{-1}(\theta^{'})}{d\theta}\Big|}{Z/Z_{\eta}} - 1 \Big) d\theta^{'}\\
=& \E_{\theta'} \Big[ f(\G^{-1}(\theta^{'})) \Big(\frac{\Big| \texttt{det}\frac{d\G^{-1}(\theta^{'})}{d\theta}\Big|}{\E_{\theta} \Big| \texttt{det}\frac{d\G^{-1}(\theta)}{d\theta}\Big|} - 1 \Big) \Big]. \\
\end{split}
\end{equation}
Here we use $$\frac{Z}{Z_{\eta}} = \E_{\theta} \Big| \texttt{det}\frac{d\G^{-1}(\theta)}{d\theta}\Big|.$$
Since we have the Lipschitz property, this leads to 
\begin{equation}
\label{eq:det-1}
\begin{split}
 \Big| \frac{\Big| \texttt{det}\frac{d\G^{-1}(\theta^{'})}{d\theta}\Big|}{\E_{\theta} \Big[ \Big| \texttt{det}\frac{d\G^{-1}(\theta)}{d\theta}\Big|\Big]} - 1  \Big| 
= &\Big|  \frac{\E_{\theta} \Big[ | \texttt{det}\frac{d\G^{-1}(\theta^{'})}{d\theta} | - | \texttt{det}\frac{d\G^{-1}(\theta)}{d\theta}|\Big] }{\E_{\theta}\Big[ | \texttt{det}\frac{d\G^{-1}(\theta)}{d\theta}|\Big]} \Big|  \\
\leq &  \frac{\E_{\theta}  \Big[\Big| | \texttt{det}\frac{d\G^{-1}(\theta^{'})}{d\theta} | - | \texttt{det}\frac{d\G^{-1}(y)}{d\theta}|\Big| \Big] + \E_{\theta} \Big[ \Big| | \texttt{det}\frac{d\G^{-1}(y)}{d\theta} | - | \texttt{det}\frac{d\G^{-1}(\theta)}{d\theta}|\Big|\Big] }{\E_{\theta} \Big[| \texttt{det}\frac{d\G^{-1}(\theta)}{d\theta}|\Big]} \\
\leq & \frac{\E_{\theta} \Big[c_0\lVert \theta' - y \rVert^{c_1} + c_0\lVert \theta - y \rVert^{c_1} \Big] }{\E_{\theta} \Big[ | \texttt{det}\frac{d\G^{-1}(\theta)}{d\theta}| \Big]} \\  
= & \frac{ c_0\lVert \theta' - y \rVert^{c_1} + c_0 \E_{\theta} \Big[ \lVert \theta - y \rVert^{c_1} \Big]}{\E_{\theta} \Big[ | \texttt{det}\frac{d\G^{-1}(\theta)}{d\theta}| \Big]}. \\
\end{split}
\end{equation}
Bringing~\cref{eq:det-1} into \cref{eq:exp_diff} leads to 
\begin{equation*}
\begin{split}
 &\Big\lVert \E_{\theta'} \Big[ f(\G^{-1}(\theta^{'})) \Big(\frac{\Big| \texttt{det}\frac{d\G^{-1}(\theta^{'})}{d\theta}\Big|}{\E_{\theta} \Big| \texttt{det}\frac{d\G^{-1}(\theta)}{d\theta}\Big|} - 1 \Big) \Big] \Big\rVert_{\infty} \\
\leq &\frac{c_0 \E_{\theta'} \Big[ \Big\lVert f(\G^{-1}(\theta^{'})) \Big\rVert_{\infty} \lVert \theta' - y \rVert^{c_1} \Big] }{\E_{\theta} \Big[ \Big| \texttt{det}\frac{d\G^{-1}(\theta)}{d\theta}\Big| \Big] }  + c_0 \E_{\theta'} \Big[ \Big\lVert f(\G^{-1}(\theta^{'})) \Big\rVert_{\infty} \Big]  \frac{ \E_{\theta} \Big[ \lVert \theta - y \rVert^{c_1} \Big] }{\E_{\theta} \Big[ \Big| \texttt{det}\frac{d\G^{-1}(\theta)}{d\theta}\Big| \Big]}  \\
\leq &\frac{c_0 \Big( \E_{\theta'} \Big[ \Big\lVert f(\G^{-1}(\theta^{'})) \Big\rVert_{\infty}^2 \Big] \E_{\theta'} \Big[ \lVert \theta' - y \rVert^{2c_1}\Big] \Big)^{1/2}}{\E_{\theta} \Big[ \Big| \texttt{det}\frac{d\G^{-1}(\theta)}{d\theta}\Big|\Big]}
+ c_0 \E_{\theta'} \Big[ \Big\lVert f(\G^{-1}(\theta^{'})) \Big\rVert_{\infty} \Big]  \frac{ \E_{\theta} \Big[ \lVert \theta - y \rVert^{c_1} \Big] }{\E_{\theta} \Big[ \Big| \texttt{det}\frac{d\G^{-1}(\theta)}{d\theta}\Big| \Big]}  \\
\leq &\frac{c_0 \Big( \E_{\theta'} \Big[ \Big\lVert f(\G^{-1}(\theta^{'})) \Big\rVert_{\infty}^2 \Big]  \Big)^{1/2}}{c_3/Z_{\eta}} 
\Big(\E_{\theta'} \Big[ \lVert \theta' - y \rVert^{2c_1}\Big]\Big)^{1/2}
+ c_0 \E_{\theta'} \Big[ \Big\lVert f(\G^{-1}(\theta^{'})) \Big\rVert_{\infty} \Big]  \frac{ \E_{\theta} \Big[ \lVert \theta - y \rVert^{c_1} \Big] }{c_3/Z_{\eta}} \\
\leq &\frac{2 c_0 (c_2 + N_y)}{c_3} \bigO \Big(\rho(\Sigma_{\eta})^{c_1} \sqrt{\texttt{det}\Sigma_{\eta}}\Big).
\end{split}
\end{equation*}
This leads to the bounds~\eqref{eq:post-pull-bound}.
\end{proof}

\section{Proof of Theorem~\ref{th:ExKI}}
\label{sec:app:ExKI:proof}
\begin{proof}
The update equations of the ExKI in~\cref{eq:ExKI-1.1,eq:ExKI-1.2} can be written as 
\begin{equation}
\label{eq:app:ExKI-1}
    \begin{split}
    &\mean_{n+1} = \mean_{n} + \pCov_{n+1} d\G (\mean_{n})^T\Big(d\G (\mean_{n})\pCov_{n+1}d\G (\mean_{n})^T + \Sigma_{\nu}\Big)^{-1}\big(y - \G(\mean_{n})\big),\\
    &\Cov_{n+1} = \pCov_{n+1} - \pCov_{n+1} d\G (\mean_{n})^T\Big(d\G (\mean_{n})\pCov_{n+1}d\G (\mean_{n})^T + \Sigma_{\nu}\Big)^{-1}
    d\G (\mean_{n})\pCov_{n+1},
    \end{split}
\end{equation}
where $\pCov_{n+1} = \Cov_{n} + \Sigma_{\omega} = 2\Cov_n$.
By applying Sherman–Morrison–Woodbury formula, the ExKI update equations~\eqref{eq:app:ExKI-1} can be rewritten as 
\begin{equation}
\label{eq:app:ExKI-2}
    \begin{split}
    &\mean_{n+1} = \mean_{n} + \Cov_{n+1} d\G (\mean_{n})^T\Sigma_{\nu}^{-1}\big(y - \G(\mean_{n})\big),\\
    &\Cov_{n+1}^{-1} = \frac{1}{2}\Cov_{n}^{-1} +  
    d\G(\mean_{n})^T \Sigma_{\nu}^{-1} d\G(\mean_{n}).
    \end{split}
\end{equation}
The stationery mean and covariance $\mean_{*}, \Cov_{*}$ of~\cref{eq:app:ExKI-2} satisfy
\begin{equation}
\label{eq:app:ExKI-3}
    \begin{split}
    &0 = \Cov_{*} d\G (\mean_{*})^T\Sigma_{\nu}^{-1}\big(y - \G(\mean_{*})\big),\\
    &\Cov_{*}^{-1} = \frac{1}{2}\Cov_{*}^{-1} +  
    d\G(\mean_{*})^T \Sigma_{\nu}^{-1} d\G(\mean_{*}).
    \end{split}
\end{equation}
Since $\Cov_{*}$ and $d\G(\mean_{*})$ are non-singular, they are uniquely determined as following, 
\begin{equation*}
    \mean_{*} = \G^{-1}(y)  \quad \textrm{and} \quad 
    \Cov_*^{-1} = d\G(\mean_{*})^T\Sigma_{\eta}^{-1}d\G(\mean_{*}).
\end{equation*}

When the extended Kalman filter is applied to estimate the mean and covariance of the pull-back random variable $\theta^{\eta} = \G^{-1}(y- \eta)$, we have
\begin{equation}
\begin{split}
&\E[\theta^{\eta}] := \G^{-1}(y- \E \eta) = \G^{-1}(y) = \mean_{*},\\
&\textrm{Cov}[\theta^{\eta}] := d\G^{-1}(y)\textrm{Cov}[\eta] d\G^{-T}(y) = d\G^{-1}(y)\Sigma_\eta d\G^{-T}(y) = \textrm{C}_{*}.
\end{split}
\end{equation}
\end{proof}

\bibliographystyle{unsrt}
\bibliography{references}
\end{document}